\documentclass[10pt]{amsart}
\usepackage{url}
\usepackage{hyperref,amssymb}
\hypersetup{colorlinks=true,linkcolor=blue,
urlcolor=cyan,citecolor=magenta} 
\usepackage[hyperpageref]{backref}
\usepackage{frcursive,calrsfs,mathrsfs}
\usepackage{xcolor}
\newtheorem{theorem}{Theorem}[section]

\newtheorem{remark}[theorem]{Remark}
\newtheorem{remarks}[theorem]{Remarks}

\newtheorem{definitions}[theorem]{Definitions}

\numberwithin{equation}{section}
\newcommand{\Q}{{\mathbb{Q}}}
\newcommand{\Z}{{\mathbb{Z}}}
\newcommand{\F}{{\mathbb{F}}}
\newcommand{\ds}{\displaystyle}

\newcommand{\ft}{\footnotesize}
\newcommand{\ns}{\normalsize}
\newcommand{\order}{\raise0.8pt \hbox{${\scriptstyle \#}$}}
\newcommand{\lien}{\mathrel{\mkern-4mu}}
\newcommand{\too}{\relbar\lien\rightarrow}

\newcommand{\plus}{\ds\mathop{\raise 0.5pt \hbox{$\bigoplus$}}\limits}
\newcommand{\prd}{\ds\mathop{\raise 2.0pt \hbox{$\prod$}}\limits}
\newcommand{\sm}{\ds\mathop{\raise 1.5pt \hbox{$\sum$}}\limits}
\newcommand{\ffrac}[2]{\hbox{\ft $\ds\frac{#1}{#2}$}}
\newcommand{\ZK}{{\bf Z}_K}
\newcommand{\Norm}{{\bf N}}
\newcommand{\Trace}{{\bf T}}
\newcommand{\BB}{{\bf B}}
\newcommand{\BX}{{\bf X}}
\newcommand{\Sgn}{{\bf S}}
\newcommand{\pgcd}{{\bf gcd}}

\newcommand{\ram}{{\rm ram}}

\newcommand{\disc}{{\rm disc}}
\newcommand{\tr}{{\rm trace}}

\newcommand{\even}{{\rm even}}
\newcommand{\odd}{{\rm odd}}
\newcommand{\CD}{{\mathcal D}}
\newcommand{\CM}{{\mathcal M}}
\newcommand{\CT}{{\mathcal T}}
\newcommand{\CH}{{\mathcal H}}

\author[Georges Gras]{Georges Gras}

\address{Villa la Gardette -- 4, chemin de Ch\^ateau Gagni\`ere --
38520 Le Bourg d'Oisans, France -- Url: {\rm \url{http://orcid.org/0000-0002-1318-4414}}}
\email{g.mn.gras@wanadoo.fr}



\begin{document}

\title[Norm of the fundamental unit of $\Q(\sqrt M)$]
{New characterization of the norm \\ of the fundamental unit of $\Q(\sqrt M)$ }

\date{May 24, 2023}

\begin{abstract}  
We give an elementary criterion for the norm of the fundamental unit $\varepsilon_K$ 
of $K=\Q(\sqrt M)$, $M$ square-free. More precisely, if $\varepsilon_K = a+b\sqrt M$, $a, b \in \Z$ 
or $\frac{1}{2}\Z$, its norm $\Sgn_K$ only depends on $m := \pgcd \big(\frac{a+1}
{\pgcd(a+1,b)}, M\big)$ and $m' := \pgcd \big(\frac{a-1}{\pgcd(a-1,b)}, M\big)$ as follows 
when $-1$ is a global norm: $\Sgn_K = -1$ if and only if $m=m'=1$ (resp. 
$m=m'=2$) for $M$ odd (resp. even) (Theorems \ref{mainresult} or \ref{maincriterion}).
\end{abstract}

\maketitle

\tableofcontents

\section{Introduction -- Main result} 
Let $K =: \Q(\sqrt M)$, $M \in \Z_{\geq 2}$ square-free, be a real quadratic field 
and let $\ZK$ be its ring of integers. Recall that $M$ is called the ``Kummer 
radical'' of $K$, contrary to any ``radical'' $R=M r^2$ giving the same field $K$.
We will write the elements of $\ZK$ under the form $\alpha = \frac{1}{2} (u + v \sqrt M)$,
with $u, v$ of same parity. We denote by $\Trace_{K/\Q} =: \Trace$ and $\Norm_{K/\Q} 
=: \Norm$, the trace and norm maps in $K/\Q$, so that $\Trace(\alpha) = u$ and 
$\Norm(\alpha) = \frac{1}{4}(u^2 - M v^2)$. 

\smallskip
We denote by $\varepsilon_K^{} >1$ the fundamental unit of $K$ and by 
$\Sgn_K =: \Sgn := \Norm(\varepsilon_K^{}) \in \{-1,1\}$ its norm. An obvious 
necessary condition for $\Sgn = -1$ is to have $-1 \in \Norm(K^\times)$,
equivalent to the fact that any odd prime ramified in $K/\Q$ is congruent 
to $1$ modulo $4$.

\medskip
The starting point of our result is the following observation proved by means of 
elementary applications of class field theory (see Theorem \ref{chevalley} 
assuming $-1 \in \Norm(K^\times)$ and Remark \ref{mn}\,(iii) in the case $-1 
\notin \Norm(K^\times)$):

\bigskip\noindent
{\it Set $M = \prod_{q \mid M} q$ for the prime divisors $q$ of $M$ and let 
${\mathfrak q} \mid q$ be the prime ideal of $K = \Q(\sqrt M)$ over $q$.
The fundamental unit $\varepsilon_K^{}$ of $K$ is of norm 
$-1$ if and only if the relation $\prod{\mathfrak q} = (\sqrt M)$ is the unique 
non-trivial relation of principality (in the ordinary sense) between the 
ramified prime ideals of $K$ (that is to say, the ${\mathfrak q}$'s dividing $(\sqrt M)$
and ${\mathfrak q}_2 \mid 2$ if $2 \nmid M$ ramifies, whence if $M \equiv 3 \pmod 4$).}

\bigskip
Of course, when $-1 \notin \Norm(K^\times)$ and when 
$2 \nmid M$ ramifies, one can verify the existence of a relation of principality, either of the 
form $\prod{\mathfrak q}^{e_q}$ (distinct from $(1)$ and $(\sqrt M)$) with exponents 
in $\{0,1\}$, or else ${\mathfrak q}_2$ principal (e.g., $M = 3 \times 17$ with ${\mathfrak q}_2$ 
principal, the ideals ${\mathfrak q}_3$ and ${\mathfrak q}_{17}$ being non-principal;
for more examples and comments, see Remark \ref{mn}\,(iii)).

\smallskip
Then we can state the main result, under the assumption $-1 \in \Norm(K^\times)$
(see Theorem \ref{maincriterion} for more details and informations):

\begin{theorem} \label{mainresult}
Let $\varepsilon_K^{} = a+b\sqrt M>1$, $a, b \in \Z$ or $\frac{1}{2}\Z$, be the 
fundamental unit of~$K$. We consider the integers $A + B\sqrt M$ and $A' + B'\sqrt M$, 
defined as follows:
\begin{equation*}
\begin{aligned}
\varepsilon_K^{} + 1 = a+1+b\sqrt M & =: g \,(A + B\sqrt M),\ \, 
\hbox{where} \  g := \pgcd\,(a+1,b) \in \Z_{>0}, \\
\varepsilon_K^{} - 1 = a-1+b\sqrt M&  =: g' (A' + B'\sqrt M), \ \, 
\hbox{where}\  g' := \pgcd\,(a-1,b) \in \Z_{>0}.
\end{aligned}
\end{equation*}
 
 \noindent
Let $m := \pgcd\,(A, M)$ and $m' := \pgcd\,(A', M)$. Then we have:

\medskip
(i) If $M$ is odd, $\varepsilon_K^{}$ is of norm $-1$ if and only if $m=m'=1$.
 
\smallskip
(ii) If $M$ is even, $\varepsilon_K^{}$ is of norm $-1$ if and only if $m=m'=2$.
\end{theorem}

Let $\CD_{\leq \BX}$ be the set of discriminants $D \leq \BX$ corresponding to
quadratic fields $K $ such that $-1 \in \Norm(K^\times)$ and let $\CD^-_{\leq \BX}$
the subset of discriminants such that $\Sgn := \Norm(\varepsilon_K^{})=-1$, and put 
$\Delta_\disc^- := \ds \lim_{\BX \to \infty} \frac{\CD^-_{\leq \BX}}{\CD_{\leq \BX}}$. 
Recently, the Stevenhagen conjecture \cite[Conjecture 1.4]{St}, for the density 
$\Delta_\disc^- = 0.5805...$, has been proved from techniques developed by 
Koymans--Pagano \cite{KoPa1,KoPa2}, in relation with that of Smith \cite{Sm}, 
based on estimations of the $2^k$-ranks of the class groups of the $K$'s, 
$k \geq 2$. It raises the question of whether our point of view can constitute 
a way to define another notion of density since it does not seem compatible 
with the common principle of classification of the fields via their discriminants; 
this will be discussed in Section \ref{densities}.

\section{Characterizations of \texorpdfstring{$\Norm(\varepsilon_K^{})=-1$}{Lg}} 

In this Section, we give (Theorems \ref{chevalley}, \ref{maincriterion}) a new 
characterization of the norm $\Sgn$ of the fundamental unit $\varepsilon_K^{}$ of 
$K=\Q(\sqrt M)$ when $-1 \in \Norm(K^\times)$, in particular for the case $\Sgn=-1$ 
giving the solvability of the norm equation $u^2 - M v^2 = -4$, $u, v \in \Z_{>0}$.

\subsection{Class field theory results involving \texorpdfstring{$\Norm(\varepsilon_K^{})$}{Lg}}

We will use the classical class field theory context given by the Chevalley--Herbrand
formula and by the standard exact sequence defining the group of invariant classes
(for simplicity, we refer to our book \cite{Gra1}, but any classical reference book may agree).
Of course, in the case of quadratic fields, the theory is equivalent to that of quadratic 
forms and goes back to Gauss and, after that in a number field context, to Kummer, 
Hilbert, Takagi, Hasse, Chevalley--Herbrand, Fr\"olich, Furuta, Ishida, Leopoldt, for 
particular formulas in the area of genus theory, but we are convinced that the general 
framework may be used, for instance, in relative quadratic extensions $K/k$, since the 
index $(E_k : \Norm E_K)$ of groups of units, when $E_k \subset \Norm(K^\times)$, is as 
mysterious as $(E_\Q : \Norm E_K)$ for $K = \Q(\sqrt M)$, when $-1 \in \Norm(K^\times)$
(for some examples in this direction, see \cite[Th\'eor\`eme 3.2, Corollaires 3.3, 3.4]{Gra2}).

\begin{theorem}\label{chevalley}
Consider Kummer radicals $M = q_1 \cdots q_r$ or $M = 2 \cdot  q_2 \cdots q_r$, 
$r \geq 1$, with odd primes $q_i \equiv 1 \pmod 4$ ($r=1$ gives $M= q_1$ or $M=2$)
and let $K = \Q(\sqrt M)$. If the prime $q$ (odd or not) divides $M$ we denote by 
${\mathfrak q}$ the prime ideal of $K $ above $q$.

\smallskip\noindent
(i) The fundamental unit $\varepsilon_K^{}$ is of norm $\Sgn = -1$ if and only 
if the relation $\prd_{q \mid M} {\mathfrak q} = (\sqrt M)$ is the unique 
relation of principality, distinct from $(1)$, between the ramified prime 
ideals of $K$ (this relation and the trivial one will be called the canonical relations).

\smallskip\noindent
(ii) In the case $\Sgn = 1$, the unique non-canonical relation ${\mathfrak m} 
= \prd_{q \mid m} {\mathfrak q}  = (\alpha)$, of support $m \mid M$
distinct from $1$ and $M$, is given by $\varepsilon_K^{}+1 =: g\, \alpha$, 
where the rational integer $g$ is maximal and, similarly, the complementary 
relation ${\mathfrak n} = \prd_{q \mid n} {\mathfrak q}  = (\beta)$, of support 
$n = \frac{M}{m}$ distinct from $1$ and $M$, is given by $\varepsilon_K^{} -1 
=: g' \, \beta$ where the rational integer $g'$ is maximal.
\end{theorem}

\begin{proof}
(i) Let $G := {\rm Gal}(K/\Q) =: \langle \sigma \rangle$. The Chevalley--Herbrand 
formula \cite[pp.\,402--405]{Che} gives, for the ordinary $2$-class group $\CH_K$, 
$\order \CH_K^G = \displaystyle \frac{2^{r-1}}{(\langle -1 \rangle : \langle -1\rangle 
\cap \Norm(K^\times))} = 2^{r-1}$ since $-1 \in \Norm(K^\times)$ by assumption. 
Moreover, we have the classical exact sequence:
\begin{equation} \label{ambiges}
1 \to \CH_K^\ram \too \CH_K^G \too \langle -1\rangle \cap 
\Norm(K^\times) / \langle \Sgn\rangle = \langle -1\rangle/\langle \Sgn\rangle \to 1, 
\end{equation}
where $\CH_K^\ram$ is the subgroup generated by the classes of 
the ramified prime ideals ${\mathfrak q}$ of $K$. Indeed, for an invariant class of an ideal 
${\mathfrak a}$, one associates with $(\alpha) := {\mathfrak a}^{1-\sigma}$, 
$\alpha \in K^\times$, the sign $\Norm(\alpha) = \pm1 \in
\langle -1\rangle \cap \Norm(K^\times) = \langle -1 \rangle$, defined modulo
$\langle \Sgn \rangle$ since $\alpha$ is defined modulo 
$E_K := \langle -1,\varepsilon_K^{} \rangle$.

\smallskip
$\bullet$ Image. Since there exists $\beta \in K^\times$ such that
$-1 = \Norm(\beta)$, the ideal $(\beta)$, being of norm $(1)$, is of the form 
${\mathfrak b}^{1-\sigma}$, giving a pre-image in $\CH_K^G$ (surjectivity). 

\smallskip
$\bullet$ Kernel. If ${\mathfrak a}^{1-\sigma} =: (\alpha)$ with $\Norm(\alpha) = \Sgn$, 
one may suppose, 
up to a unit, that $\Norm(\alpha) = 1$, whence $\alpha = \theta^{1-\sigma}$
(Hilbert's Theorem 90), and ${\mathfrak a} (\theta)^{-1}$ is an invariant ideal, product 
of a rational ideal by a product of ramified prime ideals of $K$. 

\smallskip
This yields $\order \CH_K^\ram = \ffrac{2^{r-1}}{(\langle -1 \rangle : \langle \Sgn\rangle)}$, 
where we recall that $r$ is the number of ramified prime ideals of $K$;
let's examine each case for $\Sgn$: 

\smallskip
If $\Sgn = -1$, then $\order \CH_K^\ram = 2^{r-1}$ and the unique non-trivial relation of 
principality between the ramified primes is the canonical one, $\prd {\mathfrak q} = (\sqrt M)$.

\smallskip
If $\Sgn = 1$, then $\order \CH_K^\ram =  2^{r-2}$, necessarily with $r \geq 2$, 
and another relation of principality does exist, given by a suitable product  
$\prd {\mathfrak q}^{e_q} = (\alpha)$, $e_q \in \Z$; since any ${\mathfrak q}^2$ 
is principal, one may write the relation under the 
form ${\mathfrak m} = \prd_{q \mid m} {\mathfrak q} 
= (\alpha)$ of support $m \mid M$. For $n= \frac{M}{m}$ 
we get the non-canonical analogous relation ${\mathfrak n} = 
\prd_{q \mid n} {\mathfrak q} = (\beta)$, with 
${\mathfrak m}{\mathfrak n} = (\alpha \beta)= (\sqrt M)$; which 
proves the first claim.

\smallskip
(ii) From ${\mathfrak m} = (\alpha)$ when $\Sgn = 1$, we get 
$\alpha^{1-\sigma} = \pm \varepsilon_K^k$, $k \in \Z$; since 
$\varepsilon_K^2 = \varepsilon_K^{1+\sigma+1-\sigma} =
\varepsilon_K^{1-\sigma}$, we may assume that ${\mathfrak m} = (\alpha)$
with  $\alpha^{1-\sigma} \in\{\pm 1, \pm \varepsilon_K^{}\}$. If
$\alpha^{1-\sigma} = \pm 1$, then $\alpha \in \{1, \sqrt M\} \times \Q^\times$
gives the canonical relations (absurd);
thus $\alpha^{1-\sigma} = \pm \varepsilon_K^{}$.
If for instance $\alpha^{1-\sigma} = 1$, then $(\alpha \sqrt M)^{1-\sigma} =-1$
and this gives the complementary relation ${\mathfrak n} = (\beta)$
with $\beta^{1-\sigma} = -\varepsilon_K^{}$.

Since we have $(\varepsilon_K^{}+1)^{1-\sigma} = \ffrac{\varepsilon_K^{}+1}
{\varepsilon_K^\sigma+1} = \ffrac{\varepsilon_K^{}+1}
{\varepsilon_K^{-1}+1} = \varepsilon_K^{}$ and similarly
$(\varepsilon_K^{}-1)^{1-\sigma} = -  \varepsilon_K^{}$,
the quotients $\ffrac{\varepsilon_K^{}+1}{\alpha}$ and
$\ffrac{\varepsilon_K^{}-1}{\beta}$ are invariants under $\sigma$,
hence are rational numbers; this proves the second claim of the theorem 
writing $\varepsilon_K^{}+1 = g \alpha$ and $\varepsilon_K^{}-1=g' \beta$
in an obvious manner ($g$ and $g'$ are the maximal rational integer factors 
of the quadratic integers $\varepsilon_K^{} \pm 1$).
\end{proof}

\begin{remarks}\label{mn}{\rm  
(i) If $r=1$, one finds again the well-known result $\Sgn = -1$ for the prime 
Kummer radicals $M=q \equiv 1 \pmod 4$ and for $M=2$. The properties given 
by Theorem \ref{chevalley}, due to the Chevalley--Herbrand formula, come from 
the ``product formula'' of Hasse's norm residue symbols in class field theory
\cite[Theorem II.3.4.1]{Gra1}.

\smallskip
(ii) The theorem is also to be related with that of Trotter \cite[Theorem, p.\,198]{Tro} 
leading, in another framework, to the study of the equations $m x^2 - n y^2 = \pm 4$,
equivalent to the principalities of ${\mathfrak m}$ and ${\mathfrak n}$, since, for instance,
${\mathfrak m} = (\frac{1}{2}(x+y \sqrt M))$ is equivalent to $x^2 - M y^2 = \pm 4 m$,
whence $m x^2 - n y^2 = \pm 4$; then, some {\it sufficient conditions of solvability} of these 
equations are given by means of properties of suitable quadratic residues.}
\end{remarks}

\begin{remark}\label{norm1}{\rm  
We intend to give some properties of the case $M = q_1 \cdots q_r$ odd,
when one assumes that $-1 \notin \Norm(K^\times)$ (there exists $q \mid M$ 
such that $q \equiv 3 \pmod 4$), then $\Sgn = 1$, the Chevalley--Herbrand 
formula becomes $\order \CH_K^G = \ds \frac{2^{r+\delta-1}}
{(\langle -1 \rangle : \langle -1\rangle \cap \Norm(K^\times))} = 2^{r+\delta-2}$, 
where $\delta = 1$ (resp.~$0$) if $M \equiv 3 \pmod 4$ (resp. if not) and the 
exact sequence \eqref{ambiges} becomes the isomorphism $\CH_K^\ram 
\simeq \CH_K^G$. Nevertheless, if $M \equiv 1 \pmod 4$ (thus an even 
number of primes $q \equiv 3 \pmod 4$ dividing $M$), we have $\delta = 0$, 
so that one obtains the pair of non-canonical relations of principality from the 
computation of $m$ and $n$ from $\varepsilon_K^{} \pm 1$.

\smallskip
Now, we assume that $M \equiv 3 \pmod 4$ implying the ramification of $2$ with $M$ 
odd and $\Sgn = 1$. So we have $\delta = 1$, $\order \CH_K^\ram = 2^{r-1}$
and the existence of a group of relations of principality of order $4$ between the $r+1$
ramified primes, including the canonical ones.

\smallskip
$\bullet$ If ${\mathfrak q}_2 \mid 2$ is principal, this gives the non-canonical relation,
and the process using $\varepsilon_K^{} \pm 1$ only gives the canonical relations 
$(1)$ and $(\sqrt M)$.

\smallskip
$\bullet$ If ${\mathfrak q}_2 \mid 2$ is not principal, this implies the existence of a non-canonical 
relation of principality, distinct from $(1)$ and $(\sqrt M)$. We intend to show that this 
relation is not of the form ${\mathfrak q}_2 \cdot {\mathfrak m}$ principal, for ${\mathfrak m}$ 
distinct from $(1)$ and $(\sqrt M)$. Otherwise, we have: 
\begin{equation}\label{1}
{\mathfrak q}_2  {\mathfrak m} = (\alpha) \ \, {\rm and}\ \,
{\mathfrak q}_2  {\mathfrak n} = (\beta),\ \  {\mathfrak m}{\mathfrak n} = (\sqrt M), 
\end{equation}
whence the existence of units $\varepsilon$, $\varepsilon'$ such that:
$\alpha^2  \varepsilon = 2m \ \, {\rm and}\ \, \beta^2 \varepsilon' = 2 n$;
modulo the squares of units one may choose $\alpha$ and $\beta$ such that:
$$\alpha^2 = \pm 2m \ \, {\rm or}\ \,\alpha^2 \varepsilon_K^{} =\pm 2m \ \ \ {\rm and} \ \ \
\beta^2 =\pm 2n \ \, {\rm or}\ \,\beta^2 \varepsilon_K^{} =\pm 2m.$$
The cases $\pm 2m$ and $\pm 2n$ does not hold since the Kummer radical $M$ is unique. So:
\begin{equation}\label{2}
\alpha^2  \varepsilon_K^{} = 2m \ \ \ {\rm and}\ \ \ \beta^2 \varepsilon_K^{} = 2 n.
\end{equation} 

In the same way, we have, from the principality relations \eqref{1}, $\alpha^{1-\sigma} 
= \eta$ and $\beta^{1-\sigma} = \eta'$ that we may write (using $\varepsilon_K^2 = 
\varepsilon_K^{1-\sigma}$ since $\Sgn = \varepsilon_K^{1+\sigma} = 1$):
$$\alpha^{1-\sigma} = \pm 1 \ \, {\rm or}\ \,\alpha^{1-\sigma} = \pm \varepsilon_K^{}
 \ \, {\rm and}\ \, \beta^{1-\sigma} = \pm 1 \ \, {\rm or}\ \,\beta^{1-\sigma} = \pm \varepsilon_K^{}. $$
As above, the cases $\pm 1$ are excluded (e.g., $\alpha^{1-\sigma} = \pm 1$ gives
$\alpha \in \Q^\times$ or $\Q^\times \cdot \sqrt M$). Thus:
\begin{equation}\label{3}
\alpha^{1-\sigma} = \pm \varepsilon_K^{} \ \, {\rm and}\ \, \beta^{1-\sigma} = \pm \varepsilon_K^{}.
\end{equation}

\smallskip
From \eqref{2} one gets the relation $\alpha^2 \beta^2 \varepsilon_K^2 = 2m \,2n = 4 M$
and then $\alpha \beta \varepsilon_K^{} = \pm 2 \sqrt M$, giving 
$(\alpha \beta \varepsilon_K)^{1-\sigma} = -1$. Using relations \eqref{3}, the previous relation leads to:
$$\pm \varepsilon_K^{} \cdot \varepsilon_K^{}\cdot \varepsilon_K^2 = -1\ \, {\rm (absurd)}. $$

So, when ${\mathfrak q}_2$ is non-principal, the non-canonical relations are
of the form ${\mathfrak m}$ and ${\mathfrak n}$, given by the usual process.
Let's give few examples using Program \ref{program}:

\smallskip
\ft\begin{verbatim}
M          relations                        M          relations
3 [3]       1   3   []  q2 principal        51 [3,17]   1   51  [0] q2 principal
6 [2,3]     2   3   []                      55 [5,11]   11  5   [1] q2 non principal 
7 [7]       1   7   []  q2 principal        57 [3,19]   3   19  []  D=M
11 [11]     1   11  []  q2 principal        59 [59]     59  1   []  q2 principal 
14 [2,7]    7   2   []                      62 [2,31]   2   31  []  
15 [3,5]    3   5   [1] q2 non principal    66 [2,3,11] 33  2   [0]
19 [19]     1   19  []  q2 principal        67 [67]     67  1   []  q2 principal 
21 [3,7]    3   7   []  D=M                 69 [3,23]   23  3   []  D=M
22 [2,11]   2   11  []                      70 [2,5,7]  14  5   [1] q2 non principal 
23 [23]     1   23  []  q2 principal        71 [71]     1   71  []  q2 principal 
30 [2,3,5]  5   6   [1] q2 non principal    77 [7,11]   7   11  []  D=M
31 [31]     1   31  []  q2 principal        78 [2,3,13] 26  3   [1] q2 non principal 
33 [3,11]   11  3   []  D=M                 79 [79]     1   79  [0] q2 principal 
35 [5,7]    5   7   [1] q2 non principal    83 [83]     83  1   []  q2 principal 
38 [2,19]   2   19  []                      86 [2,43]   43  2   []  
39 [3,13]   13  3   [1] q2 non principal    87 [3,29]   3   29  [1] q2 non principal
42 [2,3,7]  6   7   [1] q2 non principal    91 [7,13]   13  7   [1] q2 non principal  
43 [43]     1   43  []  q2 principal        93 [3,31]   31  3   []  D=M
46 [2,23]   2   23  []                      94 [2,47]   2   47  []
47 [47]     1   47  []  q2 principal        95 [5,19]   5   19  [1] q2 non principal 
\end{verbatim}\ns

\smallskip
The mention ${\sf D=M}$ means that $2$ does not ramify, so it does not
intervene in the relations; the mention ${\sf q2\ principal}$ is then
the unique non-canonical relation and the mention ${\sf q2\ non\ principal}$ 
occurs when $2$ ramifies but there exists a pair of non-canonical relations
$({\mathfrak m}, {\mathfrak n})$ given in the left column.
A box ${\sf [\ ] }$ means a trivial class group and ${\sf [0]}$ is
equivalent to the principality of ${\mathfrak q}_2$ when the class group
in non-trivial and ${\sf [a,\ b, ...]}$, with not all zero ${\sf a,\, b, ...}$, gives 
the components of the class of ${\mathfrak q}_2$ on the PARI basis of the 
class group. 

\smallskip
The first ``non-trivial'' example is $M=51$, where $\varepsilon_{51}^{}
= 50+7\sqrt {51}$, giving $\varepsilon_{51}^{}+1= 51 +7 \sqrt {51}$ ($m = 51$)
and $\varepsilon_{51}^{}-1 =7 \,(7 + \sqrt {51})$ ($n = 1$), thus the canonical relations.}
\end{remark}

The case $-1 \notin \Norm(K^\times)$ being without any mystery regarding
the norm $\Sgn$, we will assume $-1 \in \Norm(K^\times)$ in the sequel.

\subsection{Computation of the non-canonical relations of principality}
The following program computes, when  $-1 \in \Norm(K^\times)$ and $\Sgn = 1$, the 
non-canonical relations of principality between the ramified primes of $K$ to make statistics 
and notice that any kind of relation seems to occur with the same probability for $r$ fixed. 

\smallskip
It uses the arithmetic information given by the PARI instructions ${\sf K=bnfinit(x^2-M)}$ 
and ${\sf bnfisprincipal(K,A)}$ testing principalities. When $\Sgn = -1$, the program
gives only the trivial list ${\sf L=List([1,1,1,1])}$ corresponding to the canonical relation.
The parameter ${\sf Br}$ forces ${\sf r \geq Br}$.
As illustration, let's give examples, for $r = 5$, of the relations of principality 
${\mathfrak m} = (\alpha)$, ${\mathfrak n} = (\beta)$, by means of the exponents 
${\sf [e_1,...,e_r]}$ given by the list~${\sf L}$ and where the radical $M$ is given via the 
list of its prime divisors; for instance: 
$${\sf M=[2,5,13,17,37]\ \ L=List([0,1,0,1,1])},$$ 

\noindent
means ${\mathfrak q}_{5} \cdot {\mathfrak q}_{17} \cdot {\mathfrak q}_{37}$ principal
(or ${\mathfrak q}_{2} \cdot {\mathfrak q}_{13}$ principal).

\smallskip
\ft\begin{verbatim}
COMPUTATION OF THE RELATIONS OF PRINCIPALITY
{Br=5;for(M=2,10^6,if(core(M)!=M,next);r=omega(M);if(r<Br,next);
i0=1;if(Mod(M,2)==0,i0=2);f=factor(M);ellM=component(f,1);
for(i=i0,r,c=ellM[i];if(Mod(c,4)!=1,next(2)));
L0=List;for(i=1,r,listput(L0,0));P=x^2-M;K=bnfinit(P,1);
print();print("M=",ellM);F=idealfactor(K,M);
for(k=1,2^r-1,B=binary(k);t=#B;L=L0;for(i=1,t,listput(L,B[i],r-t+i));A=1;
for(j=1,r,e=L[j];if(e==0,next);A=idealmul(K,A,component(F,1)[j]));
Q=bnfisprincipal(K,A)[1];if(Q==0,print("L=",L))))}
\end{verbatim}\ns

\smallskip
\ft\begin{verbatim}
M=[2,5,13,17,37]   L=[0,1,0,1,1]       M=[2,5,13,17,53]   L=[0,0,1,0,1]
M=[2,5,13,29,37]   L=[0,1,1,0,0]       M=[2,5,13,37,41]   L=[0,1,1,0,1]
M=[2,5,13,17,109]  L=[0,1,0,1,1]       M=[2,5,17,37,41]   L=[0,1,0,1,0]
M=[2,5,13,29,73]   L=[0,1,1,0,0]       M=[2,5,13,41,53]   L=[0,0,1,1,1]
M=[2,5,13,37,61]   L=[0,1,1,0,0]       M=[2,5,17,29,61]   L=[0,0,0,1,1]
M=[2,5,13,17,137]  L=[0,1,1,1,0]       M=[2,5,13,17,149]  L=[0,1,0,1,1]
M=[2,5,17,29,73]   L=[0,0,1,0,1]       M=[2,5,13,41,73]   L=[0,1,1,1,0]
M=[2,5,13,29,109]  L=[0,1,0,0,1]       M=[2,5,17,41,61]   L=[0,1,0,0,1]
M=[2,5,13,37,89]   L=[0,0,1,1,1]       M=[2,5,13,17,197]  L=[0,1,0,1,1]
M=[2,5,17,29,89]   L=[0,1,0,1,1]       M=[2,5,29,37,41]   L=[0,1,1,0,1]
M=[2,5,17,29,97]   L=[0,0,1,0,1]       M=[2,5,17,29,109]  L=[0,1,0,1,0]
M=[2,5,17,53,61]   L=[0,1,1,0,1]       M=[2,5,17,37,89]   L=[0,1,0,1,1]
M=[2,5,13,29,149]  L=[0,1,0,0,1]       M=[2,5,13,61,73]   L=[0,1,1,0,0]
M=[2,5,17,37,97]   L=[0,0,1,0,1]       M=[2,5,13,53,89]   L=[0,0,1,1,1]
M=[2,5,17,37,101]  L=[0,1,0,1,0]       M=[2,5,13,17,293]  L=[0,1,0,1,1]
M=[2,5,13,29,173]  L=[0,0,1,0,1]       M=[2,5,29,37,61]   L=[0,0,1,0,1]
M=[2,5,13,37,137]  L=[0,1,1,0,1]       M=[2,5,17,29,137]  L=[0,1,1,1,0]
M=[2,5,17,41,97]   L=[0,0,1,0,1]       M=[2,5,17,37,113]  L=[0,0,1,0,1]
M=[2,5,17,29,149]  L=[0,1,0,1,0]       M=[2,5,13,29,197]  L=[0,0,1,0,1]
M=[2,5,13,61,97]   L=[0,1,1,0,0]       M=[2,5,17,29,157]  L=[0,1,1,1,0]
M=[2,5,13,17,353]  L=[0,1,1,1,0]       M=[2,13,17,29,61]  L=[0,0,0,1,1]
M=[2,5,29,37,73]   L=[0,1,1,0,1]       M=[2,5,17,53,89]   L=[0,0,0,0,1]
M=[2,5,37,41,53]   L=[0,0,1,1,1]       M=[2,5,13,17,373]  L=[0,0,1,0,1]
M=[2,5,13,37,173]  L=[0,0,1,1,0]       M=[2,5,13,73,89]   L=[0,1,1,0,1]
M=[2,5,29,41,73]   L=[0,1,1,1,0]       M=[2,5,13,37,181]  L=[0,1,0,1,0]
M=[2,5,17,29,181]  L=[0,1,0,0,1]       M=[2,5,17,53,101]  L=[0,1,1,0,1]
M=[2,5,13,73,97]   L=[0,0,0,1,1]       M=[2,5,17,61,89]   L=[0,1,0,1,1]
M=[2,5,37,41,61]   L=[0,1,0,1,1]       M=[2,5,17,37,149]  L=[0,1,0,1,0]
M=[2,13,17,41,53]  L=[0,1,0,0,1]       M=[2,5,17,29,197]  L=[0,1,0,0,1]
\end{verbatim}\ns

\subsection{Main characterization of \texorpdfstring{$\Norm(\varepsilon_K^{})$}{Lg}}
Assume that $-1 \in \Norm(K^\times)$ and let $\Sgn := \Norm(\varepsilon_K^{})$. 

\smallskip
When $\Sgn=1$, the non-canonical relation ${\mathfrak m} = \prod_{q \mid m}
{\mathfrak q} = (\alpha)$, of support $m \mid M$ distinct from $1$ and $M$, 
comes from Theorem \ref{chevalley}\,(ii); it suffices to determine the quadratic 
integer $\alpha$, without any rational factor, deduced from $\varepsilon_K^{}+1$, 
this giving ${\mathfrak m}$. In the same way $\varepsilon_K^{}-1$ gives 
${\mathfrak n} = (\beta)$ of support $n = \frac{M}{m}$ such that 
${\mathfrak m}{\mathfrak n} = (\alpha \beta) = (\sqrt M)$.

\smallskip
Let's give some numerical examples of the process leading to the result:

\medskip
(i) For $M=15170 = 2 \cdot 5 \cdot 37 \cdot 41$, $\varepsilon_K^{} = 739 + 6 \sqrt M$ and:
$$\varepsilon_K^{} + 1 = 740 + 6 \sqrt M = 2 \times (370 + 3 \sqrt M), $$ 
where $2 = \pgcd\,(740,6)$; then $\pgcd\,(370,M) = 370 = 2 \cdot 5 \cdot 37$, whence the 
principality of ${\mathfrak m} = {\mathfrak q}_2{\mathfrak q}_{5} {\mathfrak q}_{37}$,
which immediately gives the principality of ${\mathfrak q}_{41}$ that can be obtained from:
$$-\varepsilon_K^{} + 1 = -738 - 6 \sqrt M = -6 \times (123 + \sqrt M), $$ 
for which $\pgcd\,(123,M) = 41$.

\medskip
(ii) For $M= 141245 = 5 \cdot 13 \cdot 41 \cdot 53$, $\varepsilon_K^{} = 
49609+132 \sqrt M$ and:
$$\varepsilon_K^{} + 1 = 49610+132 \sqrt M = 22 \times (2255 + 6\sqrt M), $$
where $22 = \pgcd\,(49610,132)$ and $\pgcd\,(2255,M)=5 \cdot 41$ giving the
principality of ${\mathfrak q}_{5} {\mathfrak q}_{41}$, 
and the principality of ${\mathfrak q}_{13} {\mathfrak q}_{53}$, also obtained from
 $-\varepsilon_K^{} + 1= -12 \times (4134 + 11 \sqrt M)$ and $\pgcd\,(4134,M) = 13 \cdot 53$.

\medskip
(iii) For $M=999826 = 2 \cdot 41 \cdot 89 \cdot 137$, $\varepsilon_K^{} + 1$ is given by:
\ft\begin{equation*}
\begin{aligned}
&11109636935777158836160759499956087745610931184259730878643242570969499893 \\
&0608609351188823863817034706422630544237192750927410464023060264033743426 \\
+&111106036003421265074388547121779710827974912909282096975471217501055084 \\
&481117390502436801341332286005566466631729812289759396153149523058885768 \sqrt M
\end{aligned}
\end{equation*}\ns

\noindent
with the $\pgcd$ of the two coefficients equal to:
\ft\begin{equation*}
\begin{aligned}
& 21082286734619551653000708969094423248037542899079230940776043942241398 \\
=& 2\times 3 \times 43 \times 11210269457991049  \times 7289235104943832975100088612482411236091543240227619 ,
\end{aligned}
\end{equation*}\ns

\noindent
giving the integer $A + B \sqrt M \in \ZK$ equal to:
\ft\begin{equation*}
\begin{aligned}
& 5269654604181931962753271711433450204598096091653697788648087227648861999187 \\
+ & 5270113123970195868835491287611281655343354587474034791159597224413298316 \sqrt M
\end{aligned}
\end{equation*}\ns

\noindent
then $\pgcd\,(A,M) = 41\times 89 \times 137$, giving the principality of:
\ft\begin{equation*}
\begin{aligned}
\hbox{\ns${\mathfrak q}_{41} {\mathfrak q}_{89} {\mathfrak q}_{137}$\ft} =  
&(5269654604181931962753271711433450204598096091653697788648087227648861999187 \\
+ & 5270113123970195868835491287611281655343354587474034791159597224413298316\sqrt M)
\end{aligned}
\end{equation*}\ns

\noindent
or simply that of ${\mathfrak q}_{2}$.

\begin{theorem} \label{maincriterion}
Let $M \geq 2$ be a square-free integer and put $K = \Q(\sqrt M)$; we assume that 
$-1 \in \Norm(K^\times)$. Let $\varepsilon_K^{} = a+b\sqrt M$  ($a, b \in \Z$ or 
$\frac{1}{2}\Z$) be the fundamental unit of~$K$. We consider the following
integers $A + B\sqrt M$, $A' + B'\sqrt M \in \ZK$, $A, B, A', B' \in \Z$ or $\frac{1}{2}\Z$, 
where the $\pgcd$ function must be understood in $\ZK$, giving for instance, 
$\pgcd\,(\frac{1}{2}U_0, \frac{1}{2}V_0) =  \pgcd\,(U_0,V_0)$ for $U_0$ 
and $V_0$ odd; in other words one may see $g$ and $g'$ below as the maximal 
rational integer factors of the quadratic integers:\,\footnote{\,The PARI/GP $\pgcd$ function 
gives instead $\pgcd\,(\frac{1}{2}U_0, \frac{1}{2}V_0) = \frac{1}{2} \pgcd\,(U_0,V_0)$; 
this gap only occurs when $M$ is odd, in which case this does not matter for the 
computation of the odd integers $m$ and $m'$.}
\begin{equation*}
\begin{aligned}
\varepsilon_K^{} + 1 = a+1+b\sqrt M & =: g \,(A + B\sqrt M),\ \ g := \pgcd\,(a+1,b), \\
\varepsilon_K^{} - 1 = a-1+b\sqrt M&  =: g' (A' + B'\sqrt M), \ \ g' := \pgcd\,(a-1,b).
\end{aligned}
\end{equation*}
 
 \noindent
Let $m := \pgcd\,(A, M)$, $n = \frac{M}{m}$, $C = \frac{A}{m}$, $m' := \pgcd\,(A', M)$,
$n' = \frac{M}{m'}$, $C' = \frac{A'}{m'}$.

\medskip
(i) If $M$ is odd, $\varepsilon_K^{}$ is of norm $\Sgn=1$ if and only if 
$(m, m') \ne (1, 1)$. 

\smallskip
\quad \ If $M$ is even, $\varepsilon_K^{}$ is of norm $\Sgn=1$ if and only if,
either $m > 2$ or else $m' > 2$.

\smallskip\noindent
In other words, the characterization of $\Sgn = -1$ becomes:

\smallskip
\quad If $M$ is odd, $\varepsilon_K^{}$ is of norm $\Sgn=-1$ if and only if $m=m'=1$. 

\smallskip
\quad If $M$ is even, $\varepsilon_K^{}$ is of norm $\Sgn=-1$ if and only if $m=m'=2$.

\smallskip
(ii) If the above conditions giving $\Sgn=1$ hold, then 
$\prod_{q \mid m} {\mathfrak q} = (\alpha)$ and 
$\prod_{q \mid n} {\mathfrak q} = (\beta)$, with 
$\alpha = C\, m  + B \sqrt M$ and $\beta = B\, n  + C \sqrt M$ in $\ZK$.
\end{theorem}

\begin{proof}
We do not consider the particular case $r=1$ giving $\Sgn = -1$ and $\CH_K^G = \CH_K^\ram = 1$;
we prove the characterization of $\Sgn = 1$ of the statement.

\smallskip
\quad $\bullet$ Let's assume $\Sgn = 1$:

\smallskip
Recall, from Theorem \ref{chevalley}, that since $(\varepsilon_K^{}+1)^{1-\sigma} 
= \varepsilon_K^{}$, the ideal $(\varepsilon_K^{}+1)$ gives, after elimination of its 
maximal rational factor, {\it the unique non-canonical relation of principality} 
${\mathfrak m} = (\alpha)$, of support $m \mid M$, $m \ne 1, M$, 
between the ramified primes. 

\smallskip
Since $(\sqrt M \cdot (\varepsilon_K^{}+1))^{1-\sigma} = -\varepsilon_K^{}$ 
and $(\varepsilon_K^{}-1)^{1-\sigma} = -\varepsilon_K^{}$, the ideals
$(\sqrt M \cdot (\varepsilon_K^{}+1))$ (leading to ${\mathfrak n} = (\beta)$ of 
support $n =  \frac{M}{m}$) and $(\varepsilon_K^{}-1)$ (leading to 
${\mathfrak m}' = (\alpha')$ of support $m'$) yield {\it the same non-canonical 
relations} of supports $n$ and $m'$, so that $n = m'$, with $m$, $n$ distinct 
from $1$, $M$.

\smallskip
One obtains the suitable condition $m \ne 1$ and $m'=n \ne 1$ when $M$ is odd; 
if $M$ is even, the integers $m$, $m' = n$ are of different parity, distinct from $1$ 
and $M$ and such that $m m' = M >2$; so, if we assume, for instance, $m \leq 2$ 
(hence $m=2$), this implies $m' = \frac{M}{m} >2$.

\smallskip
\quad $\bullet$ Reciprocal. We assume $m \ne 1$ (resp.
$m>2$ or $m'>2$) if $M$ is odd (resp. even):

\smallskip
In the odd case, $\varepsilon_K^{} \equiv -1 \pmod {{\mathfrak m}}$, ${\mathfrak m}$ of 
support $m$, implies $\varepsilon_K^\sigma \equiv -1 \pmod {\mathfrak m}$, whence 
$\Sgn = \varepsilon_K^{1+\sigma} \equiv 1 \pmod m$ thus $\Sgn = 1$ 
 since $m \ne 1$ implies $m > 2$ in the odd case.

\smallskip
In the even case, since $m > 2$ or  $m'>2$, the same conclusion holds,
using one of the congruence $\varepsilon_K^{} \equiv -1 \pmod {{\mathfrak m}}$
or $\varepsilon_K^{} \equiv 1 \pmod {{\mathfrak m}'}$
(note that in the reciprocal one does not know if $m' = n$).

\smallskip
\quad $\bullet$ The characterization of $\Sgn = -1$ is of course obvious, 
but the precise statement is crucial for statistical interpretation, especially in the 
even case:

\smallskip
\qquad -- In the odd case, one gets $m=m'=1$. 

\smallskip
\qquad -- In the even case, one gets 
$m \leq 2 \ \, \& \ \, m' \leq 2$, but we have $m \ne 1$ and $m' \ne 1$;
indeed, let $\varepsilon_K^{} = a+b \sqrt M$, $a,b \in \Z$, where 
$a^2 - M b^2 = -1$ implies $a$ and $b$ odd (the case of $a$ is obvious,
so $a^2 \equiv 1 \pmod 8$, then $b$ even would imply $1 \equiv -1 \pmod 8$).
Since $\varepsilon_K^{} \equiv 1 \pmod {{\mathfrak q}_2}$ with $b$ odd, 
necessarily $\varepsilon_K^{} \not\equiv 1 \pmod 2$ and
$(\varepsilon_K^{} + 1) = {\mathfrak q}_2 {\mathfrak a}$, 
where ${\mathfrak a}$ is an odd ideal and, except for $M=2$, one 
gets ${\mathfrak a}\ne 1$, then $(\varepsilon_K^\sigma + 1) = 
{\mathfrak q}_2 {\mathfrak a}^\sigma$, whence $(\varepsilon_K^{} - 1) 
= {\mathfrak q}_2 {\mathfrak a}^\sigma$. Then the integers $g = \pgcd\,(a+1,b)$
and $g' = \pgcd\,(a-1,b)$ are odd, which gives $A \equiv A' \equiv 0 \pmod 2$,
hence $m=m'=2$ which is also equivalent to $b$ odd.

\smallskip
\quad $\bullet$ The last claim is immediate since $(A + B\sqrt M) = 
(C\,m + B\sqrt M)$ is invariant by $\sigma$, without any rational factor, 
which leads to $(\sqrt M\,(C\,m + B\sqrt M)) = (C\, m\sqrt M + B\,M) = 
m\,(B + C \sqrt M)$ giving rise to $(B + C \sqrt M)$ without any rational factor.
\end{proof}

\begin{remark}\label{22}
{\rm 
(i) In the case $M$ even with $\Sgn = -1$, the ideals 
$(\varepsilon_K^{} \pm 1)$ do not define invariant classes (except that of 
${\mathfrak q}_2$ in $\Q(\sqrt 2)$ since $(\varepsilon_2 + 1) =
(2+\sqrt 2) = \sqrt 2 \cdot \varepsilon_2$ and $(\varepsilon_2 - 1) = (\sqrt 2)$;
this comes from the fact that it is the unique even case where $\CH_K^G = 1$).

\smallskip
But ${\mathfrak q}_2$ is not principal since there is no non-canonical relation 
of principality in the case $\Sgn = -1$ (e.g., $M$ $\in$ $\{$$10$, $26$, $58$, $74$, 
$82$, $106$, $122$, $130$, $170$$\}$).

\smallskip
To get illustrations for $M$ even, $\Sgn = -1$, the following program computes 
the class group of $K$ (in ${\sf HK=K.clgp}$), then (in ${\sf R}$) the 
components of the (non-trivial) class of ${\mathfrak q}_2$:

\smallskip
\ft\begin{verbatim}
{forstep(M=2,10^6,2,if(core(M)!=M,next);K=bnfinit(x^2-M,1);
S=norm(K.fu[1]);if(S==1,next);q2=component(idealfactor(K,2),1)[1];
R=bnfisprincipal(K,q2)[1];print("M=",M," R=",R," HK=",K.clgp))}

M=10    R=[1]     HK=[2,[2]]        M=199810  R=[0,0,1,1] HK=[128,[16,2,2,2]]
M=82    R=[2]     HK=[4,[4]]        M=519514  R=[32,1]    HK=[128,[64,2]]
M=130   R=[0,1]   HK=[4,[2,2]]      M=613090  R=[32,0,1]  HK=[256,[64,2,2]]
M=226   R=[4]     HK=[8,[8]]        M=690562  R=[16,2]    HK=[128,[32,4]]
M=442   R=[0,1]   HK=[8,[4,2]]      M=700570  R=[8,1,0,1] HK=[128,[16,2,2,2]]
M=2210  R=[1,1,1] HK=[8,[2,2,2]]    M=720802  R=[16,0]    HK=[128,[32,4]]
M=3026  R=[2,2]   HK=[16,[4,4]]     M=776866  R=[16,0,0]  HK=[128,[32,2,2]]
\end{verbatim}\ns

But we have the cases $M$ even and $\Sgn = 1$ (e.g., $M$ $\in$ 
$\{$$34$, $146$, $178$, $194$, $386$, $410$, $466$, $482$,
$514$, $562$$\}$), where relations ${\mathfrak q}_2$ principal are 
more frequent when $r$ is small (e.g., $M$ $\in$ $\{$$34$, $146$,
$178$, $194$, $386$, $466$, $482$$\}$; for $M=410$, the relation 
is given by $m=41$, $n=10$).

\smallskip
(ii) Consider the field $L= K(\sqrt 2)$ for $M$ even, $M=:2M'$; the extension 
$L/K$ is unramified since $\Q(\sqrt{M'})/\Q$ is not ramified at $2$. 
The extension of ${\mathfrak q}_2$ in $L$ becomes the principal ideal $(\sqrt 2)$
whatever the decomposition of $2$ in $\Q(\sqrt{M'})/\Q$.
Meanwhile, if $2$ is inert in this extension (i.e., $M' \equiv 5 \pmod 8$),  
${\mathfrak q}_2$ can not be principal in $K$, otherwise, if ${\mathfrak q}_2 
= (\alpha)$, then, in $L$, $\alpha = \eta \sqrt 2$, where $\eta$ is a unit of $L$, 
and by unicity of radicals (up to $K^\times$) this is absurd and the class
of ${\mathfrak q}_2$ capitulates in $L$. 

\smallskip
So either $m=2 \ \& \ m'=2$, with ${\mathfrak q}_2$ non-principal and $\Sgn = - 1$ (e.g., 
$M$ $\in$ $\{$$10$, $26$, $58$, $74$, $82$, $106$, $122$$\}$ showing that 
reciprocal does not hold since $2$ may split in $\Q(\sqrt{M'})/\Q$, as for $M=82$), 
or else $m=2 \ \& \ m'>2$, giving the principality of ${\mathfrak q}_2$ and $\Sgn = 1$
(e.g., all the previous cases $M$ $\in$ $\{$$34$, $146$, $178$, $194$, 
$386$, $466$, $482$$\}$).}
\end{remark}

\subsection{Computation of \texorpdfstring{$\Sgn$}{Lg}
by means of the {\bf gcd} criterion}\label{program}

The following PARI/GP program computes (if any) the non-canonical relation of 
principality ${\mathfrak m}$ distinct from $(1)$, $(\sqrt M)$ between the ramified 
primes, {\it only by means of the previous result 
on the coefficients of the fundamental unit} (Theorem \ref{maincriterion}), and 
deduces the norm $S$ without calculating it; when there are no non-canonical 
relations (whence $\Sgn = -1$), the corresponding data is empty.
One has only to give the bound ${\sf BD}$ of the discriminants $D$ whose
odd prime divisors are congruent to $1$ modulo $4$. The counters ${\sf CD}$, 
${\sf Cm}$, ${\sf Cp}$, ${\sf C22}$, enumerate the sets $\CD$, $\CD^-$, 
$\CD^+$ and the set $\CD_{22}^-$ of cases where $m=m'=2$ (equivalent to $M$ 
even and $\Sgn = -1$), respectively; for short we do not write the cases where $m=1$, 
giving $\Sgn = -1$ ($M \in$ $\{$$5$, $13$, $17$, $29$, $37,\,\ldots$\}: 

\smallskip
\ft\begin{verbatim}
MAIN PROGRAM COMPUTING S VIA THE RELATIONS OF PRINCIPALITY
{BD=10^7;CD=0;Cm=0;Cp=0;C22=0;for(D=5,BD,v=valuation(D,2);if(v!=0 & v!=3,next);
i0=1;M=D;if(v==3,M=D/4;i0=2);if(core(M)!=M,next);if(Mod(M,4)==3,next);
r=omega(M);f=factor(M);ellM=component(f,1);for(i=i0,r,ell=ellM[i];
if(Mod(ell,4)==3,next(2)));CD=CD+1;res=lift(Mod(M,2));e=quadunit(D);
Y=component(e,3)/(res+1);X=component(e,2)+res*Y;g=gcd(X+1,Y);m=gcd((X+1)/g,M);
if(m==1,S=-1;print("D=",D," M=",M," relations: "," ",","," ",", S=",S);Cm=Cm+1);
if(m>2,S=1;print("D=",D," M=",M," relations: ",m,",",M/m,", S=",S);Cp=Cp+1);
if(m==2,gp=gcd(X-1,Y);mp=gcd((X-1)/gp,M);
if(mp>2,S=1;print("D=",D," M=",M," relations: ",m,",",mp,", S=",S);Cp=Cp+1);
if(mp==2,S=-1;print("D=",D," M=",M," relations: "," ",","," ",", S=",S,","," m=mp=2");
Cm=Cm+1;C22=C22+1)));
print("CD=",CD," Cm=",Cm," Cp=",Cp," C22=",C22);
print("Cm/CD=",Cm/CD+0.0," Cp/CD=",Cp/CD+0.0);
print("C22/CD=",C22/CD+0.0," C22/Cm=",C22/Cm+0.0)}

D=8 M=2       relations:   , ,   S=-1, m=mp=2
D=40 M=10     relations:   , ,   S=-1, m=mp=2
D=104 M=26    relations:   , ,   S=-1, m=mp=2
D=136 M=34    relations: 2,17,   S=1
D=205 M=205   relations: 5,41,   S=1
D=221 M=221   relations: 17,13,  S=1
D=232 M=58    relations:   , ,   S=-1, m=mp=2
D=296 M=74    relations:   , ,   S=-1, m=mp=2
D=305 M=305   relations: 5,61,   S=1
D=328 M=82    relations:   , ,   S=-1, m=mp=2
D=377 M=377   relations: 13,29,  S=1
D=424 M=106   relations:   , ,   S=-1, m=mp=2
D=488 M=122   relations:   , ,   S=-1, m=mp=2
D=505 M=505   relations: 5,101,  S=1
D=520 M=130   relations:   , ,   S=-1, m=mp=2
D=545 M=545   relations: 109,5,  S=1
D=584 M=146   relations: 73,2,   S=1
D=680 M=170   relations:   , ,   S=-1, m=mp=2
D=689 M=689   relations: 53,13,  S=1
D=712 M=178   relations: 89,2,   S=1
D=745 M=745   relations: 149,5,  S=1
D=776 M=194   relations: 2,97,   S=1
D=793 M=793   relations: 13,61,  S=1
D=808 M=202   relations:   , ,   S=-1, m=mp=2
D=872 M=218   relations:   , ,   S=-1, m=mp=2
D=904 M=226   relations:   , ,   S=-1, m=mp=2
D=905 M=905   relations: 181,5,  S=1
D=1096 M=274  relations:   , ,   S=-1, m=mp=2
D=1160 M=290  relations:   , ,   S=-1, m=mp=2
D=1192 M=298  relations:   , ,   S=-1, m=mp=2
D=1205 M=1205 relations: 5,241,  S=1
D=1256 M=314  relations:   , ,   S=-1, m=mp=2
D=1345 M=1345 relations: 269,5,  S=1
D=1384 M=346  relations:   , ,   S=-1, m=mp=2
D=1405 M=1405 relations: 5,281,  S=1
D=1448 M=362  relations:   , ,   S=-1, m=mp=2
D=1469 M=1469 relations: 13,113, S=1
D=1480 M=370  relations:   , ,   S=-1, m=mp=2
D=1513 M=1513 relations: 17,89,  S=1
D=1517 M=1517 relations: 41,37,  S=1
D=1537 M=1537 relations: 29,53,  S=1
D=1544 M=386  relations: 2,193,  S=1
D=1576 M=394  relations:   , ,   S=-1, m=mp=2
D=1640 M=410  relations: 41,10,  S=1
D=1717 M=1717 relations: 17,101, S=1
D=1768 M=442  relations:   , ,   S=-1, m=mp=2
D=1832 M=458  relations:   , ,   S=-1, m=mp=2
D=1864 M=466  relations: 233,2,  S=1
D=1885 M=1885 relations: 29,65,  S=1
D=1928 M=482  relations: 2,241,  S=1
D=1945 M=1945 relations: 389,5,  S=1
D=1961 M=1961 relations: 37,53,  S=1
(...)
D=9999592 M=2499898 relations:   , ,      S=-1, m=mp=2
D=9999617 M=9999617 relations: 21881,457, S=1
D=9999665 M=9999665 relations: 153841,65, S=1
D=9999688 M=2499922 relations:   , ,      S=-1, m=mp=2
D=9999709 M=9999709 relations: 113,88493, S=1
D=9999784 M=2499946 relations: 353,7082,  S=1
D=9999797 M=9999797 relations: 34129,293, S=1
D=9999821 M=9999821 relations: 1381,7241, S=1
D=9999845 M=9999845 relations: 1999969,5, S=1
D=9999944 M=2499986 relations: 73529,34,  S=1
D=9999953 M=9999953 relations: 270269,37, S=1
D=9999977 M=9999977 relations: 277,36101, S=1
\end{verbatim}\ns

\ft\begin{verbatim}
BD=10^7:
CD=866200 Cm=691947  Cp=174253  C22=70295
Cm/CD =0.79883052    Cp/CD =0.20116947 
C22/CD=0.08115331    C22/Cm=0.10159015
\end{verbatim}\ns

\section{Remarks on density questions}\label{densities}

A classical principle in number theory is to examine some deep invariants 
(as class groups, units, etc.) of families of fields (assuming, in general, 
that some parameters are fixed, as for instance, the Galois group, the 
signature, etc.), classified regarding the discriminants. The analytic reason
is that the order of magnitude of $\ffrac{h_K\cdot R_K}{\sqrt{D_K}}$ is controlled 
by suitable $\zeta$-functions and then, if the discriminant $D_K$ increase in the 
family, the class number $h_K$ and/or the regulator $R_K$ increase or,
at least, have a larger complexity. This case represents the reality 
quite well in a global context, but it can be questioned for $p$-adic
framework or non semi-simple Galois setting. We intend to give some
remarks in this direction about the norm of $\Sgn = \Norm(\varepsilon_K^{})$
linked significantly to the $2$-class group.

\subsection{Classical approach of the density}
We will describe the case of $\Sgn = -1$, using the following definitions:

\begin{definitions}\label{def}
(i) Denote by $\CM$ (resp. $\CD$) the set of all Kummer radicals
(resp. of all Discriminants), such that $-1 \in \Norm(K^\times)$. 
We have, from a result of Rieger \cite{Rie}:
$$\order \CM_{\leq \BX} \approx \frac{3}{2\pi}\,\prd_{p \equiv 1 \,{\rm mod}\, 4}
\!\!\Big(1- \frac{1}{p^2}\Big)^{\!\frac{1}{2}}\!\! \frac{\BX}{\sqrt{\log(\BX)}} 
\approx \frac{0.464592...  \,\BX}{\sqrt{\log(\BX)}} , \ \ 
 \ \ \order \CD_{\leq \BX} \approx \frac{3}{4}\,\order \CM_{\leq \BX}. $$

(ii) Denote by $\CM^- \subset \CM$ (resp. $\CD^- \subset \CD$) the subset of all
Kummer radicals (resp. of all Discriminants), such that $\Sgn := \Norm(\varepsilon_K^{})=-1$
and put:
$$\Delta_\disc^- := \lim_{\BX \to \infty} 
\frac{\CM^-_{\leq \BX}}{\CM_{\leq \BX}} = \lim_{\BX \to \infty} 
\frac{\CD^-_{\leq \BX}}{\CD_{\leq \BX}}. $$
\end{definitions}

Many heuristics have given $\Delta_\disc^-$ around $0.5$ or $0.6$.
Our characterization would give a density around $\frac{3}{2} \cdot \Big(\frac{6}
{\pi^2}\Big)^2 \approx 0.554363041753... $, but without a precise classification
by means of ascending discriminants or of number of ramified primes, a context 
using structure of the $2$-class group, whence quadratic symbols, R\'edei's 
matrices, Furuta symbols \cite{Fu}, etc. These classical principles consist in using the 
filtration of the $2$-class group following, e.g., the theoretical algorithm described in 
whole generality in \cite{Gra3}, with the fixed point formulas generalizing that of 
Chevalley--Herbrand (say for the quadratic case): 
$$\order (\CH_K^{i+1}/\CH_K^i) = 
\ffrac{2^{r-1}}{(\Lambda_K^i : \Lambda_K^i \cap \Norm (K^\times))}, \ \ i \geq 0, $$
where $\CH_K^i = {\rm Ker}(1-\sigma)^i$, the $\Lambda_K^i$'s, with $\Lambda_K^i 
\subseteq \Lambda_K^{i+1}$ for all $i$, defining a sequence of suitable 
subgroups of $\Q^\times$; more precisely, $\Lambda_K^0 = \langle-1\rangle$,
$\Lambda_K^1 = \langle -1, q_1, \cdots, q_r \rangle$, the next $\Lambda_K^{i+1}$'s
introducing ``random'' numbers $b = \Norm( {\mathfrak b})$ from identities of the form 
${\mathfrak a} = (y) {\mathfrak b}^{1-\sigma}$, when $x \in \Lambda_K^i$
is such that $(x) = \Norm({\mathfrak a}) = \Norm(y)$, $cl({\mathfrak a})
\in \CH_K^i$, $y \in K^\times$. 

\smallskip
Then the index $(\Lambda_K^i : \Lambda_K^i 
\cap \Norm (K^\times))$ being nothing else than  
$\order \rho_K^i(\Lambda_K^i)$, where $\rho_K^i$ is the $r$-uple of 
Hasse's norm residue symbols giving rise to generalized ``R\'edei matrices
of quadratic residue symbols'', or more simply, random maps $\F_2^{r_i} \to 
\F_2^{r-1}$ (product formula of the symbols), $r_i = \dim_{\F_2} 
\big(\Lambda_K^i/(\Lambda_K^i)^2 \big)$.

\smallskip
Similar viewpoints are linking the norm of $\varepsilon_K^{}$ to the structure of the 
$2$-class group in the restricted sense (see many practical examples in \cite{Gra2}). 

\smallskip
This proportion of $\CD^-$ inside $\CD$ (\cite[\S\,1, p.\,122]{St}, \cite[(1.3), p.\,1328]{BoSt}), 
was conjectured by Stevenhagen to be:
$$P = 1-\prd_{k \geq 0} \Big(1-\frac{1}{2^{1+2k}}\Big) \approx  0.5805775582\,\ldots $$

We refer to \cite{St}, then to \cite{FK2,KoPa2} for history and bibliographical comments about 
the norm of the fundamental unit of $\Q(\sqrt M)$ and for his heuristic based 
on the properties of densities $P_t$, corresponding to discriminants having $t$ distinct 
prime factors. 

\smallskip
These results involving the $2$-class groups structures allow informations about 
$\Delta_\disc^-$ and especially the determination of lower and upper bounds. For this 
aspect, we refer to the Chan--Koymans--Milovic--Pagano paper \cite{CKMP} 
who had proven that $\Delta_\disc^-$ is larger than $0.538220$, improving Fouvry--Kl\"uners 
results \cite{FK1,FK2} saying that $\Delta_\disc^-$ lies between $0.524275$ and 
$0.6666666$, which gave $0.538220 \leq \Delta_\disc^- \leq 0.6666666$.

For $\BX = 10^8$, a PARI \cite{P} calculation  gives the experimental density
$\ds \frac{\order \CD_{\leq \BX}^-}{\order\CD_{\leq \BX}} \approx 
0.787255$. Moreover, it is well known that such partial densities
decrease as $\BX$ increases; in other words, computer approaches 
are misleading as explained in \cite{St}. More precisely the arithmetic
function $\omega(x)$ giving the number of prime divisors of $x$ fulfills
the optimal upper bound $\omega(x) \leq (1+o(1))\,\ffrac{\log(x)}{\log(\log(x))}$
\cite[I.5.3]{Ten}, so that large Kummer radicals have ``more prime divisors'',
whence more important probability to get $\Sgn = -1$.

\smallskip
In the present paper, we do not use these classical ways, so that the main
question is to understand how densities may be defined; let's give
the example of classification by ascending traces before giving that
of the gcd criterion.

\subsection{Classification by ascending traces of fundamental units} 
In another direction, let's apply the ``First Occurrence Process'' algorithm \cite[Section 4, 
Theorem 4.6]{Gra4} in the interval $[1,\BB]$, simultaneously for the two cases $s=-1$ 
and $s=1$ of the polynomials $m_s(t) = t^2 -4s$, under the condition $-1 \in 
\Norm(K^\times)$ (to be checked only for $s=1$). 

\smallskip
Recall that in the set $\CT_{\leq \BB}$ of traces $\Trace(\varepsilon_K^{}) \leq \BB$, 
$\CT^+_{\leq \BB}$, $\CT^-_{\leq \BB}$, denote the corresponding subsets of traces 
$t \leq \BB$ of fundamental units, of norm $s = -1$, $s=1$, respectively.

\smallskip
As $\BB \to \infty$, all units are represented in these lists, as shown with the following
PARI program. Indeed, this is clear with $t^2+4 = M r^2$ giving $\CT^-_{\leq \BB}$, 
and from \cite[Theorem 4.6 and Corollaries]{Gra4}, the set $\CT^+_{\leq \BB}$ deals with 
minimal traces $t$ for which $t^2-4 = M r^2$ gives the unit $\varepsilon_K^{}$ 
(if $\Sgn = 1$) or $\varepsilon_K^2$ (if $\Sgn = -1$), but the squares of units of 
norm $-1$ were obtained with $t^2+4$ for a smaller trace, so that the process 
eliminates this data in the final list. Finally the densities ${\sf dp, dm}$ are that
of units of norms ${\sf 1, -1}$, respectively.

\smallskip
\ft\begin{verbatim}
PROGRAM FOR DENSITY OF UNITS OF NORM 1 AND OF NORM -1
{B=10^6;Cm=0;Cp=0;LM=List;for(t=1,B,mtm=t^2+4;M=core(mtm);L=List([M,-1]);
listput(LM,vector(2,c,L[c]));mtp=t^2-4;if(mtp<=0,next);M=core(mtp);
r=omega(M);i0=1;if(Mod(M,2)==0,i0=2);f=factor(M);ellM=component(f,1);
T=1;for(i=i0,r,c=ellM[i];if(Mod(c,4)!=1,T=0;break));if(T==1,L=List([M,1]);
listput(LM,vector(2,c,L[c]))));VM=vecsort(vector(B,c,LM[c]),1,8);print(VM);
for(k=1,#VM,S=VM[k][2];if(S==1,Cp=Cp+1);if(S==-1,Cm=Cm+1));
print("#VM=",#VM," Cp=",Cp," Cm=",Cm," dp=",Cp/#VM+0.," dm=",Cm/#VM+0.)}
VM=
[[2,-1],[5,-1],[10,-1],[13,-1],[17,-1],[26,-1],[29,-1],[34,1],[37,-1],[41,-1],
[53,-1],[58,-1],[61,-1],[65,-1],[73,-1],[74,-1],[82,-1],[85,-1],[89,-1],[97,-1],
[101,-1],[106,-1],[109,-1],[113,-1],[122,-1],[130,-1],[137,-1],[145,-1],[146,1],
[149,-1],[157,-1],[170,-1],[173,-1],[178,1],[181,-1],[185,-1],[194,1],[197,-1],
[202,-1],[205,1],[218,-1],[221,1],[226,-1],[229,-1],[233,-1],[257,-1],[265,-1],
[269,-1],[274,-1],[277,-1],[290,-1],[293,-1],[298,-1],[305,1],[314,-1],[317,-1],
[346,-1],[349,-1],[353,-1],[362,-1],[365,-1],[370,-1],[373,-1],[377,1],[386,1],
[389,-1],[397,-1],[401,-1],[410,1],[421,-1],[442,-1],[445,-1],[458,-1],[461,-1],
[482,1],[485,-1],[493,-1],[505,1],[509,-1],[514,1],[530,-1],[533,-1],[538,-1],
[545,1],[554,-1],[557,-1],[565,-1],[577,-1],[610,-1],[613,-1],[626,-1],[629,-1],
[653,-1],[674,1],[677,-1],[685,-1],[689,1],[697, -1],[698,-1],[701,-1],[706,1],
(...)
[906781966997,1],[906781967005,-1],[906785776013,-1],[906789585029,-1],
[906793394053,-1],[906797203085,-1],[906804821173,-1],[906808630229,-1],
[906812439293,-1],[906816248365,-1],[906820057445,-1],[906823866533,-1],
[906827675629,-1],[906831484733,-1],[906835293845,-1],[906839102965,-1],
[906842912093,-1],[906846721229,-1]]
#VM=998781   Cp=46622   Cm=952159   dp=0.0466789015   dm=0.9533210984
\end{verbatim}\ns

\smallskip
For $\BB = 10^7$, the data becomes:

\ft\begin{verbatim}
#VM=9996335  Cp=400433  Cm=9595902  dp=0.0400579812  dm=0.9599420187
\end{verbatim}\ns

\smallskip
In that context, $\varepsilon_K^{} = a + b \sqrt M$ is written $\frac{1}{2}(t + r \sqrt M)$
and only units with too large traces are missing in the above finite list. Restricting to 
the density $\Delta_\tr^-$ of the set of units of norm $-1$ inside the set of units classified 
by ascending traces, one gets:
\begin{theorem}
We have $\Delta_\tr^- :=\ds \lim_{\BB \to \infty} \frac{\order \CT^-_{\leq \BB}}
{\order \CT_{\leq \BB}} = 1$.
\end{theorem}

\begin{proof}
From  \cite[Theorem 4.5]{Gra4}, we have $\order \CT^-_{\leq \BB} \sim \BB - O(\sqrt[3]{\BB})$.
\end{proof}

\subsection{Approximations of the density from the gcd principle} \label{stat}
Theorem \ref{maincriterion}, on the characterization of the norm $\Sgn$ of 
$\varepsilon_K^{}$, allows heuristics for densities since the statement 
reduces to elementary arithmetic properties. The criterion only 
depends on properties of $m = \pgcd\,(A, M)$ and $m' = \pgcd\,(A', M)$ whose 
``probabilities'' may be computed, assuming that the the pairs $(A, M)$
and $(A', M)$ are random with Kummer radicals taken in the subset 
of square-free integers, without using the natural order of radicals or
discriminants, nor that of the order of magnitude of the integers $A$, $A'$ 
depending on the unpredictable trace of the unit.

\smallskip
Before giving some heuristics, we introduce the partial densities corresponding to 
the six cases summarized by the following array, with numerical values obtained 
for $M$ in various intervals ${\sf [bM, BM]}$:

\begin{equation*}
\begin{tabular}{|l|l|l|l|c|c|c}
\hline
\ft $M \in [bM, BM]$ \ns &  \ft $m$\ns  & \ft $m'$ \ns  
& \ft $\Sgn=\Norm(\varepsilon_K)$ \ns & \ft densities \ns & \ft $\delta$ \ns  \\  
\hline \hline
\ft  \hspace{0.4cm} even \ns & \ft $ =2$ \ns & \ft $ =2$ \ns  & 
\ft $\hspace{0.3cm} -1$ \ns  &  \ft  $\hspace{1.5cm} \downarrow$ \ns &  \ft $\delta^\even_{2,2}$ \ns    \\  
\hline 
\ft \hspace{0.4cm}  even & \ft $=2$ \ns  & \ft $>2$ \ns  &
 \ft $\hspace{0.3cm}  =1$\ns &  \ft  $\hspace{1.5cm} \downarrow$ \ns & \ft $\delta^\even_{2,m'}$ \ns   \\
 \hline 
\ft \hspace{0.4cm}  even & \ft $>2$ \ns  & \ft $=2$ \ns  &
 \ft $\hspace{0.3cm}  =1$\ns &  \ft  $\hspace{1.5cm} \downarrow$ \ns & \ft $\delta^\even_{m,2}$ \ns   \\  
\hline  
\ft  \hspace{0.4cm} even \ns & \ft $>2$\ns & \ft $>2$ \ns  & 
\ft $\hspace{0.3cm}  =1$ \ns  &  \ft  $\hspace{1.5cm} \uparrow$ \ns  & \ft $\delta^\even_{m,m'}$ \ns     \\ 
\hline \hline 
\ft \ \hspace{0.4cm}  odd \ns & \ft $ =1$ & \ft $ =1$ \ns  & 
\ft $\hspace{0.3cm} -1$ \ns  &   \ft  $\hspace{1.5cm}\downarrow $ \ns & \ft $\delta^\odd_{1,1}$ \ns     \\ 
\hline 
\ft \ \hspace{0.4cm} odd \ns & \ft $>2$ & \ft $>2$ \ns  & 
\ft $\hspace{0.3cm}  =1$ \ns  &  \ft  $\hspace{1.5cm} \uparrow$ \ns  & \ft $\delta^\odd_{m,m'}$ \ns     \\ 
\hline 
\end{tabular}
\end{equation*}

\smallskip
Notations of partial densities are given in the right column; the densities $\delta^\even_{2,m'}$ 
and $\delta^\even_{m,2}$, corresponding to each sign in formulas of $\varepsilon_K^{} \pm 1$, 
are indistinguishable about the cases $m=2\  \& \ m'>2$ or $m>2\  \& \ m'=2$, so we will only 
give the sum $\delta^\even_{2,m'}+\delta^\even_{m,2}$. Then $\delta^\even_{2,2}$ represent 
the cases $M$ even, $\Sgn = -1$ where ${\mathfrak q}_2$ is not principal. The densities 
$\delta^\odd_{1,1}$ (resp. $\delta^\odd_{m,m'}$) represent the cases $\Sgn = -1$ with no 
principality relations (resp. $\Sgn = 1$ with the two complementary principality relations).
We put $\Delta_\pgcd^- := \delta^\odd_{1,1}+\delta^\even_{2,2}$.

Many tests, using the following program, have been done and have shown that 
some densities increase while the others decrease as the  Kummer radicals are 
taken in larger intervals (indicated with arrows $\uparrow$ and $\downarrow$):

\smallskip
\ft\begin{verbatim}
{bM=2;BM=10^6;CM=0;C22=0;C2p=0;Cm2=0;Cmp=0;CC11=0;CCmp=0;
for(M=bM,BM,res8=Mod(M,8);if(res8!=1&res8!=2&res8!=5,next);
if(core(M)!=M,next);res=Mod(M,2);i0=1;if(res==0,i0=2);
r=omega(M);f=factor(M);ellM=component(f,1);for(i=i0,r,ell=ellM[i];
if(Mod(ell,4)==3,next(2)));D=M;if(res==0,D=4*M);CM=CM+1;e=quadunit(D);
res=lift(res);Y=component(e,3)/(res+1);X=component(e,2)+res*Y;
g=gcd(X+1,Y);A=(X+1)/g;m=gcd(A,M);gp=gcd(X-1,Y);Ap=(X-1)/gp;mp=gcd(Ap,M);
if(res==0,
if(m==2 & mp==2,C22=C22+1);if(m==2 & mp>2,C2p=C2p+1);
if(m>2 & mp==2,Cm2=Cm2+1);if(m>2 & mp>2,Cmp=Cmp+1));
if(res==1,
if(m==1 & mp==1,CC11=CC11+1);if(m>2 & mp>2,CCmp=CCmp+1)));
print("CM=",CM," C22=",C22," C2p=",C2p," Cm2=",Cm2," Cmp= ",Cmp,
" CC11=",CC11," CCmp= ",CCmp);
d22=C22/CM+0.0;d2p=C2p/CM+0.0;dm2=Cm2/CM+0.0;
dmp=Cmp/CM+0.0;dd11=CC11/CM+0.0;ddmp=CCmp/CM+0.0;
print("Sum=",C22+C2p+Cm2+Cmp+CC11+CCmp);
print("d22=",d22," d2p=",d2p," dm2=",dm2," dmp=",dmp,
" dd11=",dd11," ddmp=",ddmp)}
\end{verbatim}\ns

\begin{equation*} 
\begin{tabular}{|l|l|l|l|c|c|c}
\hline
\ft $M \in [1, 10^6]$ \ns &  \ft $m$\ns  & \ft $m'$ \ns  
& \ft $\Sgn=\Norm(\varepsilon_K)$ \ns & \ft densities \ns & \ft $\delta$ \ns  \\  
\hline \hline
\ft  \hspace{0.4cm} even \ns & \ft $ =2$ \ns & \ft $ =2$ \ns  & 
\ft $\hspace{0.3cm} -1$ \ns  &  \ft  $0.2347176480 \downarrow$ \ns &  \ft $\delta^\even_{2,2}$ \ns    \\  
\hline 
\ft \hspace{0.4cm}  even & \ft $ =2$, $>2$ \ns  & \ft $>2$, $=2$ \ns  &
 \ft $\hspace{0.3cm}  =1$\ns &  \ft  $0.0652421881 \downarrow$ \ns & 
 \ft $\delta^\even_{2,m'} + \delta^\even_{m,2}$ \ns   \\  
\hline  
\ft  \hspace{0.4cm} even \ns & \ft $>2$\ns & \ft $>2$ \ns  & 
\ft $\hspace{0.3cm}  =1$ \ns  &  \ft  $0.0389107558 \uparrow$ \ns  & \ft $\delta^\even_{m,m'}$ \ns     \\ 
\hline \hline 
\ft \ \hspace{0.4cm}  odd \ns & \ft $ =1$ & \ft $ =1$ \ns  & 
\ft $\hspace{0.3cm} -1$ \ns  &   \ft  $0.5475861515 \downarrow $ \ns & \ft $\delta^\odd_{1,1}$ \ns     \\ 
\hline 
\ft \ \hspace{0.4cm} odd \ns & \ft $>2$ & \ft $>2$ \ns  & 
\ft $\hspace{0.3cm}  =1$ \ns  &  \ft  $0.1135432564 \uparrow$ \ns  & \ft $\delta^\odd_{m,m'}$ \ns     \\ 
\hline 
\end{tabular}
\end{equation*}

\smallskip
\ft${\sf CM=124490}$, \par
${\sf C22=29220 C2p=4079, Cm2=4043, Cmp=4844, CC11=68169, CCmp= 14135, }$\ns \par
$\delta^\even_{2,2}+(\delta^\even_{2,m'}+ \delta^\even_{m,2})+\delta^\even_{m,m'} = 0.3388705920 \downarrow$, \par
$\delta^\even_{2,2}= 0.2347176480 \downarrow$,\ \ \ \ 
$(\delta^\even_{2,m'}+\delta^\even_{m,2})+\delta^\even_{m,m'} = 0.1041529440 \uparrow$,\par
$\delta^\odd_{1,1}+\delta^\odd_{m,m'} = 0.6611294079 \uparrow$,\ \ \ \ 
$\Delta_\pgcd^- = \delta^\even_{2,2}+\delta^\odd_{1,1} = 0.7823037995 \downarrow$.

\begin{equation*} 
\begin{tabular}{|l|l|l|l|c|c|c}
\hline
\ft $M \in [10^6, 10^7]$ \ns &  \ft $m$\ns  & \ft $m'$ \ns  
& \ft $\Sgn=\Norm(\varepsilon_K)$ \ns & \ft densities \ns & \ft $\delta$ \ns  \\  
\hline \hline
\ft  \hspace{0.4cm} even \ns & \ft $ =2$ \ns & \ft $ =2$ \ns  & 
\ft $\hspace{0.3cm} -1$ \ns  &  \ft  $0.2312433670 \downarrow$ \ns &  \ft $\delta^\even_{2,2}$ \ns    \\  
\hline 
\ft \hspace{0.4cm}  even & \ft $ =2$, $>2$ \ns  & \ft $>2$, $=2$ \ns  &
 \ft $\hspace{0.3cm}  =1$\ns &  \ft  $0.0614862378 \downarrow$ \ns & 
 \ft $\delta^\even_{2,m'} + \delta^\even_{m,2}$ \ns   \\  
\hline  
\ft  \hspace{0.4cm} even \ns & \ft $>2$\ns & \ft $>2$ \ns  & 
\ft $\hspace{0.3cm}  =1$ \ns  &  \ft  $0.0452699271 \uparrow$ \ns  & \ft $\delta^\even_{m,m'}$ \ns     \\ 
\hline \hline 
\ft \ \hspace{0.4cm}  odd \ns & \ft $ =1$ & \ft $ =1$ \ns  & 
\ft $\hspace{0.3cm} -1$ \ns  &   \ft  $0.5374200520\downarrow $ \ns & \ft $\delta^\odd_{1,1}$ \ns     \\ 
\hline 
\ft \ \hspace{0.4cm} odd \ns & \ft $>2$ & \ft $>2$ \ns  & 
\ft $\hspace{0.3cm}  =1$ \ns  &  \ft  $0.1245804159 \uparrow$ \ns  & \ft $\delta^\odd_{m,m'}$ \ns     \\ 
\hline 
\end{tabular}
\end{equation*}

\smallskip
\ft ${\sf CM=1029889}$, \par
${\sf C22=238155, C2p=31753, Cm2=31571, Cmp= 46623, CC11=553483, CCmp=128304, }$\ns \par
$\delta^\even_{2,2}+(\delta^\even_{2,m'}+ \delta^\even_{m,2})+\delta^\even_{m,m'} = 0.33767663 \downarrow$, \par
$\delta^\even_{2,2}=0.2312433670 \downarrow$, \ \ \ \ 
$(\delta^\even_{2,m'}+\delta^\even_{m,2})+\delta^\even_{m,m'} = 0.1067561648 \uparrow$,\par
$\delta^\odd_{1,1}+\delta^\odd_{m,m'} = 0.66232334 \uparrow$,\ \ \ \ 
$\Delta_\pgcd^- = \delta^\even_{2,2}+\delta^\odd_{1,1} = 0.768663419 \downarrow$.

\begin{equation*} 
\begin{tabular}{|l|l|l|l|c|c|c}
\hline
\ft $M \in [10^7, 10^8]$ \ns &  \ft $m$\ns  & \ft $m'$ \ns  
& \ft $\Sgn=\Norm(\varepsilon_K)$ \ns & \ft densities \ns & \ft $\delta$ \ns  \\  
\hline \hline
\ft  \hspace{0.4cm} even \ns & \ft $ =2$ \ns & \ft $ =2$ \ns  & 
\ft $\hspace{0.3cm} -1$ \ns  &  \ft  $0.2290897913 \downarrow$ \ns &  \ft $\delta^\even_{2,2}$ \ns    \\  
\hline 
\ft \hspace{0.4cm}  even & \ft $ =2$, $>2$ \ns  & \ft $>2$, $=2$ \ns  &
\ft $\hspace{0.3cm}  =1$\ns &  \ft  $0.058607396 \downarrow$ \ns & 
\ft $\delta^\even_{2,m'} + \delta^\even_{m,2}$ \ns   \\  
\hline 
\ft  \hspace{0.4cm} even \ns & \ft $>2$\ns & \ft $>2$ \ns  & 
\ft $\hspace{0.3cm}  =1$ \ns  &  \ft  $0.0497431369 \uparrow$ \ns  & \ft $\delta^\even_{m,m'}$ \ns     \\ 
\hline \hline 
\ft \ \hspace{0.4cm}  odd \ns & \ft $ =1$ & \ft $ =1$ \ns  & 
\ft $\hspace{0.3cm} -1$ \ns  &   \ft  $0.5297912763 \downarrow $ \ns & \ft $\delta^\odd_{1,1}$ \ns     \\ 
\hline 
\ft \ \hspace{0.4cm} odd \ns & \ft $>2$ & \ft $>2$ \ns  & 
\ft $\hspace{0.3cm}  =1$ \ns  &  \ft  $0.1327683993 \uparrow$ \ns  & \ft $\delta^\odd_{m,m'}$ \ns     \\ 
\hline 
\end{tabular}
\end{equation*}

\smallskip
\ft ${\sf CM=9652809}$, \par
${\sf C22=2211360, C2p=283421, Cm2=282305, Cmp= 480161, CC11=5113974, CCmp=1281588, }$\ns \par
$\delta^\even_{2,2}+(\delta^\even_{2,m'}+ \delta^\even_{m,2})+\delta^\even_{m,m'} = 0.3374403243 \downarrow$, \par
$\delta^\even_{2,2}=0.2290897913 \downarrow$,\ \ \ \ 
$(\delta^\even_{2,m'}+\delta^\even_{m,2})+\delta^\even_{m,m'} = 0.1083505329 \uparrow$, \par
$\delta^\odd_{1,1}+\delta^\odd_{m,m'} = 0.6625596756 \uparrow$,\ \ \ \ 
$\Delta_\pgcd^- = \delta^\even_{2,2}+\delta^\odd_{1,1} = .7588810676 \downarrow$.

\begin{equation*} 
\begin{tabular}{|l|l|l|l|c|c|c}
\hline
\ft $M \in [10^8, 10^8+10^6]$ \ns &  \ft $m$\ns  & \ft $m'$ \ns  
& \ft $\Sgn=\Norm(\varepsilon_K)$ \ns & \ft densities \ns & \ft $\delta$ \ns  \\  
\hline \hline
\ft  \hspace{0.4cm} even \ns & \ft $ =2$ \ns & \ft $ =2$ \ns  & 
\ft $\hspace{0.3cm} -1$ \ns  &  \ft  $0.2285032150 \downarrow$ \ns &  \ft $\delta^\even_{2,2}$ \ns    \\  
\hline 
\ft \hspace{0.4cm}  even & \ft $ =2$, $>2$ \ns  & \ft $>2$, $=2$ \ns  &
\ft $\hspace{0.3cm}  =1$\ns &  \ft  $0.0571186698 \downarrow$ \ns & 
\ft $\delta^\even_{2,m'} + \delta^\even_{m,2}$ \ns  \\  
\hline 
\ft  \hspace{0.4cm} even \ns & \ft $>2$\ns & \ft $>2$ \ns  & 
\ft $\hspace{0.3cm}  =1$ \ns  &  \ft  $0.0518491039 \uparrow$ \ns  & \ft $\delta^\even_{m,m'}$ \ns     \\ 
\hline \hline 
\ft \ \hspace{0.4cm}  odd \ns & \ft $ =1$ & \ft $ =1$ \ns  & 
\ft $\hspace{0.3cm} -1$ \ns  &   \ft  $0.5276699767 \downarrow $ \ns & \ft $\delta^\odd_{1,1}$ \ns     \\ 
\hline 
\ft \ \hspace{0.4cm} odd \ns & \ft $>2$ & \ft $>2$ \ns  & 
\ft $\hspace{0.3cm}  =1$ \ns  &  \ft  $0.1348590343 \uparrow$ \ns  & \ft $\delta^\odd_{m,m'}$ \ns     \\ 
\hline 
\end{tabular}
\end{equation*}

\smallskip
\ft ${\sf CM=105132}$, \par
${\sf C22=24023, C2p=2971, Cm2=3034, Cmp= 5451, CC11=55475, CCmp=14178, }$\ns \par
$\delta^\even_{2,2}+(\delta^\even_{2,m'}+\delta^\even_{m,2})+\delta^\even_{m,m'} = 0.3374709887 \downarrow$, \par 
$\delta^\even_{2,2}=0.2285032150 \downarrow$,\ \ \ \ 
$(\delta^\even_{2,m'}+\delta^\even_{m,2})+\delta^\even_{m,m'} = 0.1089677737 \uparrow$, \par
$\delta^\odd_{1,1}+\delta^\odd_{m,m'} = 0.662529011 \uparrow$,\ \ \ \ 
$\Delta_\pgcd^- = \delta^\even_{2,2}+\delta^\odd_{1,1} = 0.7561731917 \downarrow$.

\begin{equation*} 
\begin{tabular}{|l|l|l|l|c|c|c}
\hline
\ft $M \in [10^9, 10^9+10^6]$ \ns &  \ft $m$\ns  & \ft $m'$ \ns  
& \ft $\Sgn=\Norm(\varepsilon_K)$ \ns & \ft densities \ns & \ft $\delta$ \ns  \\  
\hline \hline
\ft  \hspace{0.4cm} even \ns & \ft $ =2$ \ns & \ft $ =2$ \ns  & 
\ft $\hspace{0.3cm} -1$ \ns  &  \ft  $0.2262061480 \downarrow$ \ns &  \ft $\delta^\even_{2,2}$ \ns    \\  
\hline 
\ft \hspace{0.4cm}  even & \ft $ =2$, $>2$ \ns  & \ft $>2$, $=2$ \ns  &
\ft $\hspace{0.3cm}  =1$\ns &  \ft  $0.0563555787 \downarrow$ \ns & 
\ft $\delta^\even_{2,m'} + \delta^\even_{m,2}$ \ns   \\  
\hline 
\ft  \hspace{0.4cm} even \ns & \ft $>2$\ns & \ft $>2$ \ns  & 
\ft $\hspace{0.3cm}  =1$ \ns  &  \ft  $0.0540844730 \uparrow$ \ns  & \ft $\delta^\even_{m,m'}$ \ns     \\ 
\hline \hline 
\ft \ \hspace{0.4cm}  odd \ns & \ft $ =1$ & \ft $ =1$ \ns  & 
\ft $\hspace{0.3cm} -1$ \ns  &   \ft  $0.5226055410 \downarrow $ \ns & \ft $\delta^\odd_{1,1}$ \ns     \\ 
\hline 
\ft \ \hspace{0.4cm} odd \ns & \ft $>2$ & \ft $>2$ \ns  & 
\ft $\hspace{0.3cm}  =1$ \ns  &  \ft  $0.1407482589 \uparrow$ \ns  & \ft $\delta^\odd_{m,m'}$ \ns     \\ 
\hline 
\end{tabular}
\end{equation*}

\smallskip
\ft ${\sf CM=99511}$, \par
${\sf C22=22510, C2p=2808, Cm2=2800, Cmp=5382, CC11=52005, CCmp=14006, }$\ns \par
$\delta^\even_{2,2}+(\delta^\even_{2,m'}+\delta^\even_{m,2})+\delta^\even_{m,m'} = 0.3366461999 \downarrow$, \par 
$\delta^\even_{2,2}=0.2262061480 \downarrow$,\ \ \ \ 
$(\delta^\even_{2,m'}+\delta^\even_{m,2})+\delta^\even_{m,m'} = 0.1104400517 \uparrow$, \par
$\delta^\odd_{1,1}+\delta^\odd_{m,m'} = 0.6633538000 \uparrow$,\ \ \ \ 
$\Delta_\pgcd^- = \delta^\even_{2,2}+\delta^\odd_{1,1} = 0.748811689 \downarrow$.

\begin{equation*}
\begin{tabular}{|l|l|l|l|c|c|c}
\hline
\ft $M \in [10^{10}, 10^{10}+10^6]$ \ns &  \ft $m$\ns  & \ft $m'$ \ns  
& \ft $\Sgn=\Norm(\varepsilon_K)$ \ns & \ft densities \ns & \ft $\delta$ \ns  \\  
\hline \hline
\ft  \hspace{0.4cm} even \ns & \ft $ =2$ \ns & \ft $ =2$ \ns  & 
\ft $\hspace{0.3cm} -1$ \ns  &  \ft  $0.2252429443 \downarrow$ \ns &  \ft $\delta^\even_{2,2}$ \ns    \\  
\hline 
\ft \hspace{0.4cm}  even & \ft $ =2$, $>2$ \ns  & \ft $>2$, $=2$ \ns  &
\ft $\hspace{0.3cm}  =1$\ns &  \ft  $0.0535799257 \downarrow$ \ns & 
\ft $\delta^\even_{2,m'} + \delta^\even_{m,2}$ \ns   \\  
\hline 
\ft  \hspace{0.4cm} even \ns & \ft $>2$\ns & \ft $>2$ \ns  & 
\ft $\hspace{0.3cm}  =1$ \ns  &  \ft  $0.0571751842 \uparrow$ \ns  & \ft $\delta^\even_{m,m'}$ \ns     \\ 
\hline \hline 
\ft \ \hspace{0.4cm}  odd \ns & \ft $ =1$ & \ft $ =1$ \ns  & 
\ft $\hspace{0.3cm} -1$ \ns  &   \ft  $0.5164271590 \downarrow $ \ns & \ft $\delta^\odd_{1,1}$ \ns     \\ 
\hline 
\ft \ \hspace{0.4cm} odd \ns & \ft $>2$ & \ft $>2$ \ns  & 
\ft $\hspace{0.3cm}  =1$ \ns  &  \ft  $0.1475747866 \uparrow$ \ns  & \ft $\delta^\odd_{m,m'}$ \ns     \\ 
\hline 
\end{tabular}
\end{equation*}

\smallskip
\ft ${\sf CM=94569}$, \par
${\sf C22=21301, C2p=2494, Cm2=2573, Cmp= 5407, CC11=48838, CCmp=13956, }$\ns \par
$\delta^\even_{2,2}+(\delta^\even_{2,m'}+\delta^\even_{m,2})+\delta^\even_{m,m'} = 0.3359980543 \downarrow$, \par 
$\delta^\even_{2,2}=0.2252429443 \downarrow$,\ \ \ \ 
$(\delta^\even_{2,m'}+\delta^\even\delta^\even_{m,2})+\delta^\even_{m,m'} = 0.1107551099 \uparrow$, \par
$\delta^\odd_{1,1}+\delta^\odd_{m,m'} = 0.6640019456 \uparrow$,\ \ \ \ 
$\Delta_\pgcd^- = \delta^\even_{2,2}+\delta^\odd_{1,1} = 0.7416701033 \downarrow$.

\begin{equation*}
\begin{tabular}{|l|l|l|l|c|c|c}
\hline
\ft $M \in [10^{11}, 10^{11}+2.5 \!\cdot\!10^5]$ \ns &  \ft $m$\ns  & \ft $m'$ \ns  
& \ft $\Sgn=\Norm(\varepsilon_K)$ \ns & \ft densities \ns & \ft $\delta$ \ns  \\  
\hline \hline
\ft  \hspace{0.4cm} even \ns & \ft $ =2$ \ns & \ft $ =2$ \ns  & 
\ft $\hspace{0.3cm} -1$ \ns  &  \ft  $0.2240462381 \downarrow$ \ns &  \ft $\delta^\even_{2,2}$ \ns    \\  
\hline 
\ft \hspace{0.4cm}  even & \ft $ =2$, $>2$ \ns  & \ft $>2$, $=2$ \ns  &
\ft $\hspace{0.3cm}  =1$\ns &  \ft  $0.0514559794 \downarrow$ \ns &
\ft $\delta^\even_{2,m'} + \delta^\even_{m,2}$ \ns   \\  
\hline 
\ft  \hspace{0.4cm} even \ns & \ft $>2$\ns & \ft $>2$ \ns  & 
\ft $\hspace{0.3cm}  =1$ \ns  &  \ft  $0.0607678259 \uparrow$ \ns  & \ft $\delta^\even_{m,m'}$ \ns     \\ 
\hline \hline 
\ft \ \hspace{0.4cm}  odd \ns & \ft $ =1$ & \ft $ =1$ \ns  & 
\ft $\hspace{0.3cm} -1$ \ns  &   \ft  $0.5127736860 \downarrow $ \ns & \ft $\delta^\odd_{1,1}$ \ns     \\ 
\hline 
\ft \ \hspace{0.4cm} odd \ns & \ft $>2$ & \ft $>2$ \ns  & 
\ft $\hspace{0.3cm}  =1$ \ns  &  \ft  $0.1509562704 \uparrow$ \ns  & \ft $\delta^\odd_{m,m'}$ \ns     \\ 
\hline 
\end{tabular}
\end{equation*}

\smallskip
\ft ${\sf CM=49829}$, \par
${\sf C22=11164, C2p=1308, Cm2=1256, Cmp=3028, CC11=25551, CCmp=7522, }$\ns \par
$\delta^\even_{2,2}+(\delta^\even_{2,m'}+\delta^\even_{m,2})+\delta^\even_{m,m'} = 0.3362700435 \downarrow\uparrow$, \par 
$\delta^\even_{2,2}=0.2240462381 \downarrow$,\ \ \ \ 
$(\delta^\even_{2,m'}+\delta^\even_{m,2})+\delta^\even_{m,m'} = 0.1122238054 \uparrow$, \par
$\delta^\odd_{1,1}+\delta^\odd_{m,m'} = 0.6637299564 \uparrow\downarrow$,\ \ \ \ 
$\Delta_\pgcd^- = \delta^\even_{2,2}+\delta^\odd_{1,1} = 0.7368199241 \downarrow$.

\subsubsection{First heuristic from the above data}
It is difficult to go further because of the execution time, but some  
rules appear, that are not proved, but allow possible heuristics:

\smallskip
$\bullet$ $\delta^\even := \delta^\even_{2,2}+\delta^\even_{2,m'}
+\delta^\even_{m,2}+\delta^\even_{m,m'} \to \ffrac{1}{3}$;

$\bullet$ $\delta^\odd := \delta^\odd_{1,1} + \delta^\odd_{m,m'} \to \ffrac{2}{3}$; 

\smallskip\noindent
this is almost obvious since random Kummer radicals $M$, such that 
$-1 \in \Norm(K^\times)$, are in the classes $1$, $2$ or $5$ modulo $8$, 
whence with uniform repartition $\ffrac{1}{3}$ for $M$ even (class of $2$) 
and $\ffrac{2}{3}$ for the case $M$ odd (classes of $1$ and $5$).

\smallskip
The indications $\uparrow$ and $\downarrow$ give some interesting phenomena:

\smallskip
$\bullet$ $\delta^\even_{m,m'}$ must have a hight increasing, since $\delta^\even_{2,m'}$ 
and $\delta^\even_{m,2}$ are decreasing, but the  sum $\delta^\even_{2,m'}+\delta^\even_{m,2}
+\delta^\even_{m,m'}$ is increasing.

\smallskip
$\bullet$ $\delta^\even_{2,2}$ must have a hight decreasing, since
$\delta^\even = \delta^\even_{2,2}+\delta^\even_{2,m'}+\delta^\even_{m,2}+\delta^\even_{m,m'}$ 
is decreasing while the partial sum $\delta^\even_{2,m'}+\delta^\even_{m,2}+\delta^\even_{m,m'}$ 
is increasing.

\smallskip
$\bullet$ The sum $\Delta_\pgcd^- = \delta^\odd_{1,1}+\delta^\even_{2,2}$ is much decreasing.

\smallskip
$\bullet$ The quotient $\ffrac{\delta^\even_{2,2}}{\delta^\odd_{1,1}}$ seems to be increasing  and
the quotient $\ffrac{\delta^\odd_{1,1}}{\delta^\odd_{m,m'}}$ seems to be rapidly decreasing.

\smallskip
$\bullet$ The quotient $\ffrac{\delta^\even_{2,2}}
{\delta^\even_{2,m'}+\delta^\even_{m,2}+\delta^\even_{m,m'}}$ 
seems to be  decreasing up to a constant $\rho \approx 2$.

\smallskip
To reinforce this last heuristic, let's consider the computation of the parity of the component $b$ 
of $\varepsilon_K = a + b\sqrt M$, for $M$ even and $-1 \in \Norm(K^\times)$; indeed, recall that 
in that cases $\Sgn = -1$ is equivalent to $b$ odd. The following program examines this question 
taking ${\sf M \equiv 2 \pmod 8}$, in the intervals ${\sf [k*10^7, (k+1)*10^7]}$, ${\sf k \geq 1}$:

\smallskip
\ft\begin{verbatim}
{for(k=0,50,bM=k*10^7;BM=bM+10^7;CM=0;CP=0;CI=0;forstep(M=bM+2,BM,8,
if(core(M)!=M,next);r=omega(M);f=factor(M);ellM=component(f,1);for(i=2,r,
ell=ellM[i];if(Mod(ell,4)==3,next(2)));CM=CM+1;e=quadunit(4*M);
Y=component(e,3);if(Mod(Y,2)==0,CP=CP+1);if(Mod(Y,2)==1,CI=CI+1));
print("k=",k," CM=",CM," CP=",CP," CI=",CI," rho=",CI/CP+0.0))}
k   CM       CP      CI                    rho
0   390288   122913  267375  2.1753191281638231920138634644016499475
1   373804   119118  254686  2.1380983562517839453315200053728235867
2   368305   117555  250750  2.1330441070137382501807664497469269704
3   364879   117142  247737  2.1148435232452920387222345529357531884
4   362321   116164  246157  2.1190472091181433146241520608794463001
5   360344   115669  244675  2.1153031495041886763091234470774364782
6   358748   115431  243317  2.1078999575503980733078635721773180515
7   357459   115340  242119  2.0991763481879660135252297555054621120
8   356212   114959  241253  2.0986003705668977635504832157551822824
9   355175   114509  240666  2.1017212620842029884113912443563388031
10  354229   114366  239863  2.0973278771662906807967402899463127153
11  353458   113972  239486  2.1012704874881549854350191275050012284
12  352631   113688  238943  2.1017433678136654704102455844064457111
13  351987   113671  238316  2.0965417740672642978420177529888889866
14  351367   113626  237741  2.0923116188196363508351961698906940313
15  350705   113407  237298  2.0924457925877591330341160598552117594
16  350238   113342  236896  2.0900989924299906477739937534188562051
17  349650   112900  236750  2.0969884853852967227635075287865367582
18  349183   112760  236423  2.0966920893934019155728981908478183753
19  348671   112975  235696  2.0862668732020358486390794423545032087
20  348262   112572  235690  2.0936822655722559783960487510215684184
21  347881   112435  235446  2.0940632365366656290301062836305420910
22  347434   112209  235225  2.0963113475746152269425803634289584614
23  347070   112571  234499  2.0831208748256655799451013138374892290
24  346697   112323  234374  2.0866073733785600455828280939789713594
25  346469   112613  233856  2.0766341363785708577162494561018710096
26  346065   112086  233979  2.0874953160965687061720464643220384348
27  345684   111830  233854  2.0911562192613788786551014933381024770
28  345428   111827  233601  2.0889498958212238547041412181315782414
29  345159   111590  233569  2.0930997401200824446635003136481763599
30  344767   112211  232556  2.0724884369625081320013189437755656754
31  344530   112094  232436  2.0735811015754634503184827020179492212
32  344346   111732  232614  2.0818923853506605090752872945977875631
33  343975   111378  232597  2.0883567670455565731113864497477060102
34  343833   111710  232123  2.0779070808343031062572732969295497270
35  343605   111150  232455  2.0913630229419703103913630229419703104
36  343315   111288  232027  2.0849238013083171590827402774782546187
37  343185   111320  231865  2.0828692058929213079410707869205892921
38  342857   111017  231840  2.0883288145058864858535179296864444184
39  342577   111231  231346  2.0798698204637196465014249624654997258
40  342591   111345  231246  2.0768422470699178229826215815707934797
41  342144   111123  231021  2.0789665505790880375799789422531789099
42  342078   111320  230758  2.0729249011857707509881422924901185771
43  341911   110739  231172  2.0875391686758955742782578856590722329
44  341767   110973  230794  2.0797311057644652302812395807989330738
45  341313   110928  230385  2.0768877109476417135439203807875378624
46  341356   110701  230655  2.0835855141326636615748728557104271868
47  341177   110800  230377  2.0792148014440433212996389891696750903
48  341010   110847  230163  2.0764026090015967955831010311510460364
49  340781   110667  230114  2.0793371104303902699088255758265788356
50  340700   110734  229966  2.0767424639225531453754041215886719526
\end{verbatim}\ns

\smallskip
Many oscillations can be observed, on the successive intervals ${\sf [k*10^7, (k+1)*10^7]}$,
which support the idea of a slow convergence of ${\sf CI/CP}$ towards ${\sf \rho = 2^+}$.

\smallskip
With the same principle, but in intervals of the form ${\sf [k*10^9, k*10^9+10^7]}$, $k \geq 1$ 
(long calculation time), we obtain, as expected, a slow global decreasing of the data:

\smallskip
\ft\begin{verbatim}
1  334971   109266  225705  2.0656471363461644061281642963044313876
2  329775   107868  221907  2.0572088107687173211703192791189231283
3  326962   106943  220019  2.0573483070420691396351327342603068925
4  324916   107023  217893  2.0359455444156863477943993347224428394
5  323273   105998  217275  2.0498028264684239325270288118643747995
\end{verbatim}\ns

\smallskip
To conclude, one may say that these heuristics are compatible with the next one
suggesting that $\delta^\odd_{1,1} \to \big(\frac{6}{\pi^2}\big)^{\!2} \approx 0.3695...$, then 
$\Delta_\pgcd^- := \delta^\odd_{1,1} + \delta^\even_{2,2} \to \big(\frac{6}{\pi^2}\big)^{\!2}(1+ 0.5)
\approx 0.5543...$.

\subsubsection{Second heuristic from randomness of $A$, $A'$}
In our viewpoint, we are reduced to use the well-known fact that the density of 
pairs of independent co-prime integers is $\ffrac{6}{\pi^2} \approx 0.6079...$;
but this implies that no condition is assumed, especially for a radical $R$ taken 
at random, and we must take into account that a radical lies in the subset of 
square-free integers, whose relative density is also given by $\ffrac{6}{\pi^2}$.
The integers $a \pm 1$ giving $A$, $A'$ may be considered as random and 
independent regarding~$M$.

\smallskip
$\bullet$ When $M$ is odd, $\Sgn = -1$ is equivalent to $m := \pgcd\,(A,M) = 1$;
which gives the partial density $\delta^\odd_{1,1} = \Big(\ffrac{6}{\pi^2}\Big)^{\!2}$. 

\smallskip
$\bullet$ When $M$ is even, $\Sgn = -1$ is equivalent to $m' := 
\pgcd\,(A',M) = 2$ and this may be written $\ffrac{m'}{2} := \pgcd\,\big(\frac{A'}{2},
\frac{M}{2}\big) = 1$; we know this is equivalent to $b$ odd and that a good
heuristic is that this occur with probability $\frac{1}{2}$.

\smallskip
The specific case $m=2$ (with the alternative $m'>2$ or $m'=2$) occurs
only for $M$ even, whence a coefficients $\ffrac{1}{2}$ for the corresponding
densities; indeed, $M$ even and  

\smallskip
We then obtain the following discussion, on $m$ and $m'$, from Theorem 
\ref{maincriterion}:

\smallskip
$\bullet$ $m=1$ ($M$ of any parity); thus $\Sgn = -1$ with density $\Big(\ffrac{6}{\pi^2}\Big)^{\!2}$;

\smallskip
$\bullet$ $m=m'=2$ ($M$ even); thus $\Sgn = -1$ with density 
$\ffrac{1}{2} \Big(\ffrac{6}{\pi^2}\Big)^{\!2}$.

\smallskip
Taking into account the previous interpretation, we propose, for $\CD^-$ inside $\CD$, the 
conjectural density:
\begin{equation}\label{0.5543}
\Delta_\pgcd^- = \Big(\ffrac{6}{\pi^2}\Big)^{\!2} \,\Big(1+ \ffrac{1}{2} \Big) \approx 0.554363041753. 
\end{equation}

To be compared with the density $0.5805775582$ \cite[Conjecture 1.4]{St},
proven in \cite{KoPa1,KoPa2} depending on the natural order of discriminants.

\smallskip
Our method, assumes the independence and randomness of the parameters;
moreover it does not classify the radicals (or discriminants), nor the units
(for example by means of their traces), so one can only notice that the density 
\eqref{0.5543} may be an lower bound for the classical case; fortunately it is greater 
than the lower bound $0.538220$, which was proved in \cite{CKMP}, and also less 
than $0.666666...$, proved in \cite{FK1,FK2}.

\end{document}